\documentclass[paper=a4,headsepline,DIV13]{scrartcl}

\usepackage[utf8]{inputenc}
\usepackage{amsmath}
\usepackage{amssymb}
\usepackage{amsthm}
\usepackage{enumerate}
\usepackage{graphicx}
\usepackage{booktabs}
\usepackage{subfig}
\usepackage{color}
\usepackage{mathtools}
\usepackage{subfig}

\renewenvironment{abstract}{\minisec{Abstract}}{\par\vspace{.1in}}
\newenvironment{keywords}{\minisec{Key Words}}{\par\vspace{.1in}}
\newenvironment{AMS}{\minisec{AMS subject classification}}{\par\vspace{.1in}}

\theoremstyle{plain}
\newtheorem{theorem}{Theorem}[section]

\newtheorem{lemma}[theorem]{Lemma}

\theoremstyle{remark}
\newtheorem{remark}[theorem]{Remark}

\usepackage{pst-node,pst-eucl,pstricks-add}
\usepackage[off,crop=on]{auto-pst-pdf}

\newcommand{\TheTitle}{Finite element error estimates for normal derivatives on boundary concentrated meshes}

\newcommand{\lnh}{\lvert\ln h\rvert}

\title{\TheTitle}
\author{Johannes Pfefferer \thanks{Technical University of Munich, Department
    of Mathematics, Chair of Optimal Control, Boltzmannstra\ss e 3, 85748 Garching by Munich,
    Germany. pfefferer@ma.tum.de} 
\and Max Winkler\thanks{Chemnitz University of Technology, Faculty of
  Mathematics, Professorship Numerical Mathematics (Partial Differential
  Equations), Stra\ss e der Nationen 62, 09111 Chemnitz, Germany. max.winkler@mathematik.tu-chemnitz.de}
}
\date{}
\DeclareMathOperator*{\diam}{diam}
\DeclareMathOperator*{\dist}{dist}
\DeclareMathOperator*{\supp}{supp}
\DeclareMathOperator*{\essinf}{ess\,inf}
\DeclareMathOperator*{\esssup}{ess\,sup}

\newcommand{\C}{\mathcal C}

\newcommand{\e}{e}

\begin{document}

\maketitle

\begin{abstract}
This paper is concerned with approximations and related discretization error estimates for the normal derivatives of solutions of linear elliptic partial differential equations. In order to illustrate the ideas, we consider the Poisson equation with homogeneous Dirichlet boundary conditions and use standard linear finite elements for its discretization. The underlying domain is assumed to be polygonal but not necessarily convex. Approximations of the normal derivatives are introduced in a standard way as well as in a variational sense. On general quasi-uniform meshes, one can show that these approximate normal derivatives possess a convergence rate close to one in $L^2$ as long as the singularities due to the corners are mild enough. Using boundary concentrated meshes, we show that the order of convergence can even be doubled in terms of the mesh parameter while increasing the complexity of the discrete problems only by a small factor.
As an application, we use these results for the numerical analysis of Dirichlet boundary control problems, where the control variable corresponds to the normal derivative of some adjoint variable.
\end{abstract}

\begin{keywords}
 finite element error estimates, local mesh refinement, boundary concentrated meshes, Dirichlet boundary control, surface flux, normal derivatives
\end{keywords}

\begin{AMS}
  35J05, 49J20, 65N15, 65N30
\end{AMS}

\section{Introduction}
The main purpose of this paper is to investigate convergence properties of two types of approximations to the normal derivative of the weak solution $u$ of the Poisson equation
\[
-\Delta u = f\quad\mbox{in}\quad \Omega,\qquad u=0\quad\mbox{on}\quad\partial\Omega,
\]
posed in polygonal domains $\Omega$. The first approximation, denoted by $\partial_n u_h$, is defined in a classical way, whereas the second one, denoted by $\partial_n^hu_h$, is introduced in a variational sense. Both of them require the knowledge about discrete solutions $u_h$ to the Poisson equation. In this regard and also to illustrate the ideas, we choose standard linear finite elements for the discretization.

In the recent past, error estimates for the two different approximations have been established. In general, the quality of the estimates do not only depend on the regularity of the solution but also on the structure of the underlying computational meshes. We emphasize that in the present case of polygonal domains the regularity of the solution may additionally be lowered by the appearance of corner singularities even though the input datum $f$ is arbitrarily smooth. In the following, for a concise discussion of the results from literature, we assume that the regularity of the solution is only limited by the singular terms coming from the corners and not by rough input data.

On general quasi-uniform meshes, the classical and the discrete variational
normal derivative converge in $L^2(\partial\Omega)$ with the rate $s=1$ (up to
logarithmic factors), provided that the largest interior angle $\omega$ in the
domain is less than $2\pi/3$. For larger interior angles the convergence rate
is reduced due to the corner singularities. More precisely, the convergence
rate fulfills $s<\pi/\omega - 1/2$. The corresponding results for the
classical approximation $\partial_n u_h$ of the normal derivative have been
discussed in \cite{HMW14,MW12}, whereas related results for the discrete
variational normal derivative $\partial_n^h u_h$ can be found in
\cite{AMPR16}. On certain superconvergence meshes, where, roughly speaking,
neighboring elements almost form a parallelogram, the convergence rate for the
discrete variational normal derivative $\partial_n^h u_h$ can be improved to
$s=3/2$ (again up to logarithmic factors) if the largest interior angle is less than $\pi/2$. Otherwise, the convergence rate $s$ satisfies the condition from before, cf. \cite{AMPR16}. The convergence rates $s$ for the different approximations of the normal derivatives are illustrated in Figure \ref{fig1} depending on the largest interior angle $\omega$ and the structure of the underlying computational meshes.

\psset{yunit=0.9, xunit=0.9}
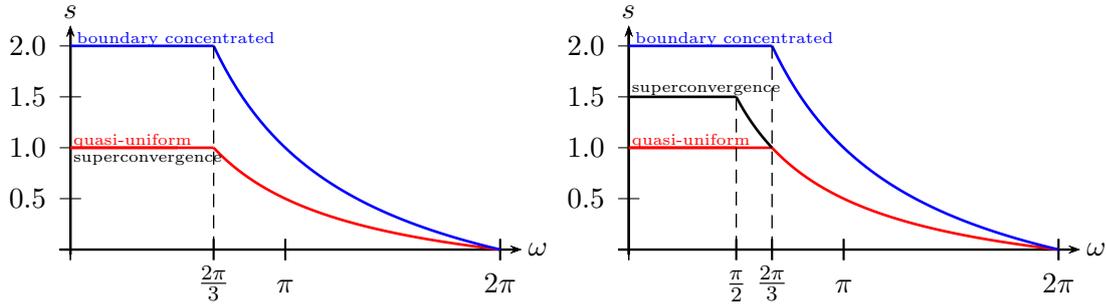
\begin{figure}
	\subfloat{
		\begin{pspicture}[showgrid=false](-1,-1)(9.4,3)
		\psline(2.094,-4pt)(2.094,4pt)
		\psline(3.1415927,-4pt)(3.1415927,4pt)
		\psline(6.283,-4pt)(6.283,4pt)
		\uput[-0](-0.3,3.5){$s$}
		\uput[-0](6.5,0){$\omega$}
		\rput[t](2.094,-6pt){$\frac{2\pi}{3}$}
		
		\psaxes[yunit=1.5,xunit=3.141,showorigin=false,trigLabels,Dy=0.5]{->}(0,0)(-0.05,-0.05)(2.1,2.2)
		\psplot[yunit=1.5,linecolor=red,linewidth=1pt]{0}{1.571}{1}
		\psplot[yunit=1.5,linecolor=red,linewidth=1pt]{0}{2.094}{1}
		\psplot[yunit=1.5,linecolor=red,linewidth=1pt]{2.094}{6.283}{3.141 x div 0.5 sub}
		\psplot[yunit=1.5,linecolor=blue,linewidth=1pt]{0}{2.094}{2}
		\psplot[yunit=1.5,linecolor=blue,linewidth=1pt]{2.094}{6.283}{6.283 x div 1 sub}
		\psline[yunit=1.5,linestyle=dashed,linewidth=0.5pt](2.094,0)(2.094,2)
		\uput[-0](-0.15,1.6){\color{red}{\tiny quasi-uniform}}
		\uput[-0](-0.15,1.32){\color{black}{\tiny superconvergence}}
		\uput[-0](-0.1,3.10){\color{blue}{\tiny boundary concentrated}}
		\end{pspicture}
	}	
	\subfloat{
		\begin{pspicture}[showgrid=false](-1,-1)(9.4,3)
		\psline(1.571,-4pt)(1.571,4pt)
		\psline(2.094,-4pt)(2.094,4pt)
		\psline(3.1415927,-4pt)(3.1415927,4pt)
		\psline(6.283,-4pt)(6.283,4pt)
		\uput[-0](-0.3,3.5){$s$}
		\uput[-0](6.5,0){$\omega$}
		\rput[t](1.571,-8pt){$\frac{\pi}{2}$}
		\rput[t](2.094,-6pt){$\frac{2\pi}{3}$}

		\psaxes[yunit=1.5,xunit=3.141,showorigin=false,trigLabels,Dy=0.5]{->}(0,0)(-0.05,-0.05)(2.1,2.2)
		\psplot[yunit=1.5,linecolor=red,linewidth=1pt]{0}{1.571}{1}
		\psplot[yunit=1.5,linecolor=red,linewidth=1pt]{0}{2.094}{1}
		\psplot[yunit=1.5,linecolor=red,linewidth=1pt]{2.094}{6.283}{3.141 x div 0.5 sub}
		\psplot[yunit=1.5,linecolor=black,linewidth=1pt]{0}{1.571}{1.5}
		\psplot[yunit=1.5,linecolor=blue,linewidth=1pt]{0}{2.094}{2}
		\psplot[yunit=1.5,linecolor=black,linewidth=1pt]{1.571}{2.094}{3.141 x div 0.5 sub}
		\psplot[yunit=1.5,linecolor=blue,linewidth=1pt]{2.094}{6.283}{6.283 x div 1 sub}
		\psline[yunit=1.5,linestyle=dashed,linewidth=0.5pt](1.571,0)(1.571,1)
		\psline[yunit=1.5,linestyle=dashed,linewidth=0.5pt](1.571,1.15)(1.571,1.5)
		\psline[yunit=1.5,linestyle=dashed,linewidth=0.5pt](2.094,0)(2.094,2)
		\uput[-0](-0.15,1.6){\color{red}{\tiny quasi-uniform}}
		\uput[-0](-0.15,2.35){\color{black}{\tiny superconvergence}}
		\uput[-0](-0.1,3.10){\color{blue}{\tiny boundary concentrated}}
		\end{pspicture}	
	}
	\caption{\label{fig1} Convergence rates for $\partial_n u_h$ and $\partial_n^hu_h$ depending on $\omega$ for the different types of meshes.}
\end{figure}

As the quantities of interest live on the boundary, it might be promising to appropriately refine the mesh towards the boundary.
In this regard, we consider a certain class of boundary concentrated meshes. These are isotropically refined towards the boundary such that the element diameter at the boundary is of order $h^2$ with $h$ denoting the maximal element diameter in the interior of the domain. As we will see, the number of elements corresponding to such meshes is of order $h^{-2}\left|\ln h\right|$ and no longer of order $h^{-2}$ as in case of quasi-uniform meshes. However, with this slight increase in the number of elements, it is possible to double the convergence rates of the two different approximations in terms of the maximal element diameter $h$ (compared to general quasi-uniform meshes). More precisely, the convergence rate $s$ in $L^2(\partial\Omega)$ is two (up to logarithmic factors) as long as the largest interior angle is less than $2\pi/3$. For larger interior angles we obtain a rate $s$ fulfilling $s<2(\pi/\omega-1/2)$, see Figure~\ref{fig1} for an illustration.

Our proof of error estimates for the approximating normal derivatives heavily
relies on finite element error estimates in weighted $L^2(\Omega)$-norms. Thereby, the weight is a regularized distance function with respect to
the boundary. In order to bound these finite element errors on graded meshes
appropriately, one has to be able to handle the weights within the
estimates. This requires to establish regularity results in weighted Sobolev
spaces with the aforementioned regularized distance function as weight. Based
on this, the weighted $L^2$-errors can then be treated by an adapted duality
argument employing a dyadic decomposition of the domain with respect to its
boundary and local energy norm estimates on the subsets. These techniques are
known for instance from maximum norm error estimates \cite{ASR09,SW79,Sir10}
or from finite element error estimates on the boundary for the Neumann problem
\cite{APR13,Win15}. However, in all these references the weights, and hence the
dyadic decomposition used in the proofs, are related to the corners of the polygonal domain.

As an application of the discrete variational normal derivative
$\partial_n^hu_h$ we consider Dirichlet boundary control problems with $L^2(\partial\Omega)$-regularization, where this
type of normal derivative naturally arises in the discrete optimality system. 
In the last decade, Dirichlet
boundary control problems have been under active research. We start with mentioning the contribution \cite{CR06}, where a control constrained problem subject to a semilinear elliptic equation is considered. There, a convergence order of $s<\min(1,\pi/2\omega)$ is proved for the error of the controls in $L^2(\partial\Omega)$. This means a rate close to one is only possible if the largest interior angle is less than $\pi/2$. However, non-convex domains are excluded in that reference. Later on, in \cite{MRV13}, comparable results for the controls are provided in case of linear problems without control constraints. In addition, the authors of this reference show that the states exhibit better convergence properties. The proof relies on a duality argument and estimates for the controls in weaker norms than $L^2(\partial\Omega)$. We note that, to the best of our knowledge, such an argumentation is restricted to problems without control constraints. For a certain time, this was the state of the art. Nevertheless, numerical experiments indicated that the controls converge with an order close to one also for larger interior angles, and can even achieve a rate close to $3/2$ if the underlying meshes satisfy certain superconvergence properties.
For smooth domains $\Omega$, where no corner singularities appear, 
these convergence rates for general and
superconvergence meshes are shown in \cite{DGH09}.
Therein, the domain is approximated by a sequence of polygons, on which
the discrete approximations are posed.
The first contributions dealing with quasi-optimal convergence rates for quasi-uniform and superconvergence meshes in
polygonal domains are \cite{AMPR15} and \cite{AMPR16}. More precisely, in \cite{AMPR15}, accurate regularity results are derived for the solution of the optimal control problem. In \cite{AMPR16}, these are applied within the proofs of the error estimates for the control. The rates of convergence for the controls in the unconstrained case coincide with those from above for the discrete variational normal derivative $\partial_n^h u_h$, as such an error for the adjoint state is one of the dominating error contributions. In these references, the control constrained case is discussed is as well. While in convex domains the error estimates for the controls are similar to those in the unconstrained case (depending on the specific choice of the control bounds), in non-convex domains the convergence rates are considerably larger. This is due to a smoothing effect of the control bounds on the continuous solution. For a more detailed discussion, we refer to the introduction of \cite{AMPR16}. In the present paper, we only consider the case without control constraints. However, we notice that the estimates can be extend to the control constrained case as well. In the unconstrained case, if we use boundary concentrated meshes, we obtain a rate $s$ of two (up to logarithmic factors) as long as the interior angles are less than $2\pi/3$. Otherwise, we get the reduced rate $s<2(\pi/\omega-1/2)$. This is quite natural as we have already observed that the error for the discrete variational normal derivative is the limiting term within the error estimates.

Finally, we notice that there is an alternative approach to the $L^2(\partial\Omega)$-regularization. 
Several articles, for instance \cite{JSW18,OPS10,OPS13}, consider a regularization in the
$H^{1/2}(\partial\Omega)$-norm instead. This guarantees a higher regularity of the
solution. In numerical experiments it turns out that the convergence rate in case of a standard discretization on quasi-uniform meshes seems to be
one order higher than for the $L^2(\partial\Omega)$-regularization in case of quasi-uniform meshes, this is
$s=2$ (up to logarithmic factors) if $\omega < 2\pi/3$, and $s<\pi/\omega+1/2$
if $\omega\ge 2\pi/3$. To the best of our knowledge, the corresponding estimates in the literature show a lower rate of convergence. This is mainly due to the fact that either standard techniques are used to bound the error for the discrete variational normal derivative or lower regularity of the data is assumed. A proof of the convergence rates stated above will be subject of a forthcoming article.

The paper is organized as follows: In Section \ref{sec:regularity}, we introduce the variational formulation to the Poisson equation and establish regularity results in weighted Sobolev spaces, where the weight is a regularized distance function with respect to the boundary. Moreover, we collect several regularity results in different weighted Sobolev spaces from the literature for the later error analysis. The discretization of the Poisson equation and the boundary concentrated meshes are introduced at the beginning of Section \ref{sec:weighted_estimate}. Moreover, the discretization error estimates in weighted $L^2(\Omega)$-norms are proven in this section. These are applied in Section \ref{sec:surfaceflux} in order to derive the error estimates for the two different approximations to the normal derivative of the solution of the Poisson equation. In addition, numerical experiments are included in this section which underline the theoretical findings. The numerical analysis for Dirichlet boundary control problems is outlined in Section \ref{sec:control}. Moreover, numerical examples are presented which exactly show the convergence rates from the theory.

In the following $c$ will denote a generic constant which is always independent of the mesh parameter $h$. We will use the notation $a\sim b$ to indicate that $a\le c b$ and $b\le c a$.

\section{Weighted regularity for elliptic problems}\label{sec:regularity}

Let us first introduce some notation which is used in this paper. 
We consider computational domains $\Omega\subset\mathbb R^2$ that are bounded by a polygon $\Gamma:=\partial\Omega$.
The corner points are numerated counter-clockwise and are denoted by $\boldsymbol c_j$, $j\in\mathcal C:=\{1,\ldots,d\}$. The interior angle at a corner point $\boldsymbol c_j$ is denoted by $\omega_j$. The index set $\mathcal{C}_{non}$ collects all indices $j$ with $\omega_j>\pi$, i.e., the indices corresponding to non-convex corners.
The boundary edge having endpoints $\boldsymbol c_j$ and $\boldsymbol c_{j+1}$ ($\boldsymbol c_{d+1}:=\boldsymbol c_1$) is denoted by $\Gamma_j$, $j\in\mathcal C$.
The classical Sobolev spaces are denoted as usual by $W^{k,p}(\Omega)$ for $k\in\mathbb N_0$, $p\in[1,\infty]$, and by $H^k(\Omega)$ in case of $p=2$.
The corresponding norms are denoted by $\|\cdot\|_{W^{k,p}(\Omega)}$ and  $\|\cdot\|_{H^{k}(\Omega)}$, respectively. By $H^k_0(\Omega)$ we denote the completion of $C_0^\infty(\Omega)$ functions with respect to $\|\cdot\|_{H^{k}(\Omega)}$.
Moreover, we use the notation $\left\|\cdot\right\|_{L^2(\Omega)}$ and $(\cdot,\cdot)_{L^2(\Omega)}$ for the norm and the inner product in $L^2(\Omega)=H^0(\Omega)$. An analogous notation is used for the spaces defined on the boundary.

For $f\in L^2(\Omega)$ we consider the Poisson equation in variational form:
\begin{equation}\label{eq:weak_form_1}
\mbox{Find}\ u\in H_0^1(\Omega)\colon\quad (\nabla u,\nabla v)_{L^2(\Omega)} = (f,v)_{L^2(\Omega)}\qquad \forall v\in H_0^1(\Omega).
\end{equation}
We introduce the regularized distance function 
\begin{equation}\label{eq:weight}
\sigma(x) := d_I + \rho(x)\quad\mbox{with}\quad \rho(x):=\dist(x,\Gamma):=\inf_{y\in\Gamma} |x-y|,
\end{equation}
and some number $d_I\in(0,\e^{-1})$ exactly specified later.
In the following we investigate regularity results in weighted spaces containing $\sigma$ as weight function.
\begin{lemma}\label{lem:weighted_reg_h1}
 There exists a constant $c>0$ independent of $d_I$ such that 
  \begin{align*}
   &(i) \quad&  \|\sigma^{-1} u\|_{L^2(\Omega)} + \|\nabla u\|_{L^2(\Omega)} &\le c \|\sigma f\|_{L^2(\Omega)},\\[.5em]
   &(ii) \quad&  \|\sigma^{-1/2} u\|_{L^2(\Omega)} + |\ln d_I|\|\sigma^{1/2}\nabla u\|_{L^2(\Omega)} &\le c |\ln d_I|^2\|\sigma^{3/2} f\|_{L^2(\Omega)}, \quad\mbox{if}\ \Omega\ \mbox{is convex}.
  \end{align*}  
\end{lemma}
 \begin{figure}
  \begin{center}
  \includegraphics[width=.5\textwidth]{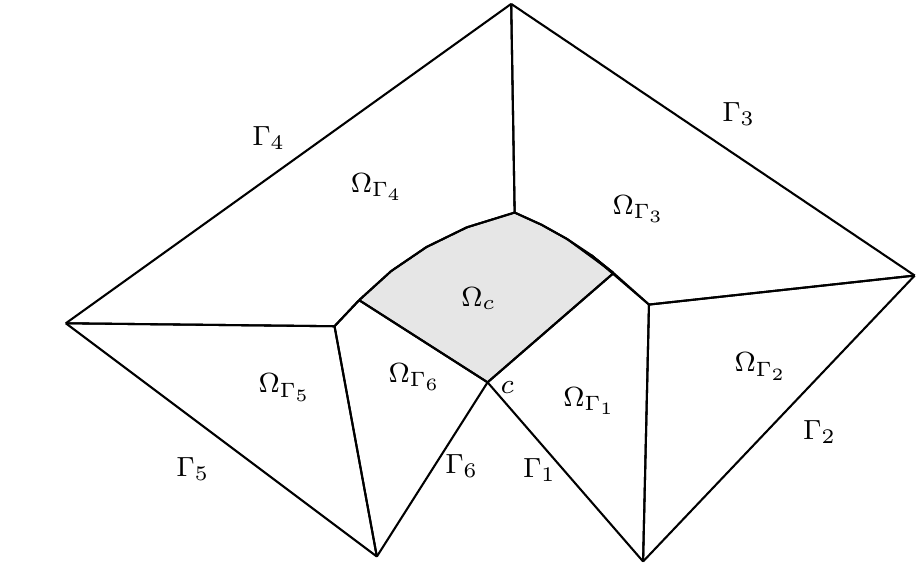}
  \end{center}
  \caption{Decomposition of $\Omega$ into the sets $\Omega_i$, $i\in\mathcal C$, and $\Omega_{\boldsymbol c}$.}
  \label{fig:decomposition}
 \end{figure}
 \begin{proof}
 As illustrated in Figure~\ref{fig:decomposition} we associate to each edge $\Gamma_i$, $i\in\mathcal C$, 
 the subsets
 \[
  \Omega_{\Gamma_i}:=\{x\in \Omega\colon \rho(x) = \dist(x,\Gamma_i)\},
 \]
 and to each non-convex corner $\boldsymbol c_j$ with $j\in\mathcal C_{\text{non}}$ the subsets
 \[
  \Omega_{\boldsymbol c_j}:=\{x\in\Omega\colon \rho(x) = |x-\boldsymbol c_j|\}
 \]
such that
\[
	\bar \Omega=\left(\bigcup_{i=1}^d\bar\Omega_{\Gamma_i}\right)\cup\left(\bigcup_{j\in\mathcal C_{\text{non}}}\bar\Omega_{\boldsymbol c_j}\right).
\]
 
 For each set $\Omega_{\Gamma_i}$, we introduce the local coordinates
 $(x_i,y_i)^\top = F_i^{-1}(x,y)$ with an affine linear map $F_i(x_i,y_i):=
 B_i (x_i, y_i)^\top + b_i$. Here, $B_i\in\mathbb R^{2\times 2}$ is a rotation matrix and $b_i\in\mathbb R^2$ a translation vector
 chosen in such a way that $F_i(0,0) = \boldsymbol c_i$ 
 and $F_i(|\Gamma_i|,0) = \boldsymbol c_{i+1}$.
 Moreover, we introduce the bounds $\overline x_i$ and $\overline y_i(x_i)$ such that $\Omega_{\Gamma_i}$ can be parameterized in local coordinates by
 \begin{equation}\label{eq:def_omega_i}
  \Omega_{\Gamma_i} = \{(x,y)^\top = F_i(x_i,y_i) \in\mathbb R^2\colon x_i\in (0,\overline x_i),\ y_i\in (0,\overline y_i(x_i))\}.
 \end{equation}
 The regularized distance function satisfies for $(x,y)\in\Omega_{\Gamma_i}$
 \begin{equation}\label{eq:weight_local}
 \sigma(x,y) = d_I + y_i(x,y)\quad\mbox{and}\quad \nabla \sigma(x,y)=B_i \begin{pmatrix}0\\1\end{pmatrix}.
 \end{equation}
 Moreover, we write $u_{\Gamma_i}(x_i,y_i) = u(F_i(x_i,y_i))$ and confirm that
 \[
 	\nabla_{\Gamma_i} u_{\Gamma_i}(x_i,y_i)=B_i^\top\nabla u(F_i(x_i,y_i)) \quad \text{with }\nabla_{\Gamma_i} = (\partial/\partial x_i, \partial/\partial y_i)^\top.
 \]
 
 To describe the sets $\Omega_{\boldsymbol c_j}$, we instead use polar coordinates $r_j(x,y)$ and $\varphi_j(x,y)$ located at the corner $c_j$ such that $(r_j(\boldsymbol c_{j+1}),\varphi_j(\boldsymbol c_{j+1}))^\top=(|\Gamma_j|,0)^\top$.
 Then, we find a representation of the form
 \[
  \Omega_{\boldsymbol c_j} = \{(x,y)^\top=(r_j\cos\varphi_j,r_j\sin\varphi_j)^\top\in \mathbb R^2 \colon \varphi_j\in\left(\frac\pi2,\omega_j-\frac\pi2\right),\ r_j\in (0,\bar r_j(\varphi_j))\}
 \]
 with an appropriate function $\bar r_j$.
 Within $\Omega_{\boldsymbol c_j}$ we write $u_{\boldsymbol c_j}(r_j,\varphi_i) = u(r_j\cos\varphi_j,r_j\sin\varphi_j)$. There is the relation
 \[
	 \nabla_{\boldsymbol c_j}u_{\boldsymbol c_j}(r_j,\varphi_j)= \left(\begin{array}{cc}\cos\varphi_j & \sin\varphi_j\\ -r_j\sin\varphi_j & r_j\cos\varphi_j\end{array}\right)\nabla u(r_j\cos\varphi_j,r_j\sin\varphi_j)
 \]
 with $\nabla_{\boldsymbol c_j}=(\partial/\partial r_j, \partial/\partial \varphi_j)^\top$.
 Moreover, the weight function $\sigma$ possesses for $(x,y)\in\Omega_{\boldsymbol c_j}$ the representation
 \[\sigma(x,y) = d_I + r_j(x,y).\]

 The result $(i)$ follows from the weak formulation of \eqref{eq:weak_form_1} and the Cauchy-Schwarz inequality:
 \begin{equation}\label{eq:estimate_gradient}
  \|\nabla u\|_{L^2(\Omega)}^2 = (\nabla u,\nabla u)_{L^2(\Omega)} = (f,u)_{L^2(\Omega)} \le \|\sigma f\|_{L^2(\Omega)} \|\sigma^{-1} u\|_{L^2(\Omega)}.
 \end{equation}
 Once we have shown $\|\sigma^{-1} u\|_{L^2(\Omega)}\leq c\|\nabla u\|_{L^2(\Omega)}$, the result is proven. For that purpose, we consider the sub-domains $\Omega_{\Gamma_i}$ and $\Omega_{\boldsymbol c_i}$ separately.
 Integration by parts using the local coordinates $(x_i,y_i)$ and the fact that $u|_\Gamma\equiv 0$ implies together with the Cauchy-Schwarz inequality
 \begin{align*}
 &\frac12\|\sigma^{-1}u\|_{L^2(\Omega_{\Gamma_i})}^2 = \frac12\int_0^{\bar x_i} \int_0^{\bar y_i(x_i)} \frac{u_{\Gamma_i}^2(x_i,y_i)}{(d_I+y_i)^2} \mathrm dy_i \mathrm dx_i \\
  &\quad= -\frac12\int_0^{\bar x_i} \left. \frac{u_{\Gamma_i}^2(x_i,y_i)}{d_I+y_i} \right|_{y_i=0}^{\bar y_i(x_i)} \mathrm dx 
  + \int_0^{\bar x_i}\int_0^{\bar y_i(x_i)} \frac{u_{\Gamma_i}(x_i,y_i) \partial_{y_i} u_{\Gamma_i}(x_i,y_i)}{d_I+y_i}\mathrm dy_i \mathrm dx_i \\
  &\quad \le \|\sigma^{-1} u\|_{L^2(\Omega_{\Gamma_i})}\|\nabla_{\Gamma_i} u_{\Gamma_i}\|_{L^2(\Omega_{\Gamma_i})}=\|\sigma^{-1} u\|_{L^2(\Omega_{\Gamma_i})}\|\nabla u\|_{L^2(\Omega_{\Gamma_i})}.
 \end{align*}
 In case of the sub-domains $\Omega_{\boldsymbol c_j}$, we first use the property $\sigma(x) \ge r_j(x)$. In a second step we enlarge the domain to a circular sector with radius $\hat r_j=\max_{\varphi_j}\bar r_j(\varphi_j)$ and $\varphi_j\in(0,\omega_j)$ containing $\Omega_{\boldsymbol c_j}$. Afterwards, we use the fact that $u|_\Gamma\equiv 0$ in combination with a Poincar\'e type inequality on the enlarged domain. This leads to
 \begin{align*}
  \|\sigma^{-1} &u \|_{L^2(\Omega_{\boldsymbol c_j})}^2\leq \int_{\frac{\pi}{2}}^{\omega_j-\frac{\pi}{2}}\int_0^{\bar r_j(\varphi_j)} r_j^{-1}u_{\boldsymbol c_j}(r_j,\varphi_i)^2 \mathrm d r_j\mathrm d\varphi_j\le \int_{0}^{\omega_j}\int_0^{\hat r_j} r_j^{-1}u_{\boldsymbol c_j}(r_j,\varphi_i)^2 \mathrm d r_j\mathrm d\varphi_j\\  
  &\le c\int_0^{\hat r_j} \int_{0}^{\omega_j} r_j^{-1}\left(\partial_{\varphi_j}u_{\boldsymbol c_j}(r_j,\varphi_j)\right)^2 \mathrm d\varphi_j\mathrm d r_j \le c\int_0^{\hat r_j} \int_{0}^{\omega_j} r_j|\nabla u(r_j\cos\varphi_j,r_j\sin\varphi_j)|^2 \mathrm d\varphi_j\mathrm d r_j \\
  &\le c\|\nabla u\|_{L^2(\Omega)}^2.
 \end{align*}  
 Summation over all subsets $\Omega_{\Gamma_i}$, $i\in\mathcal C$, and $\Omega_{\boldsymbol c_j}$, $j\in\mathcal C_{\text{non}}$, yields the desired estimate
\begin{equation}\label{eq:poincare}
  \|\sigma^{-1}u\|_{L^2(\Omega)} \le c \|\nabla u\|_{L^2(\Omega)}.
\end{equation}

 To show the second estimate $(ii)$, we apply the Leibniz rule:
 \begin{equation}\label{eq:weighted_h1_reg}
  \|\sigma^{1/2}\nabla u\|_{L^2(\Omega)}^2 = \int_\Omega \sigma \nabla u\cdot \nabla u = \int_\Omega \nabla u\cdot\nabla (\sigma u) - \int_\Omega u\nabla u\cdot\nabla\sigma.
 \end{equation}
 The variational formulation \eqref{eq:weak_form_1} with $v:=\sigma u$ leads to
 \begin{equation}\label{eq:weighted_h1_reg_1}
  \int_\Omega \nabla u\cdot\nabla (\sigma u) = (f,\sigma u)_{L^2(\Omega)} \le c\|\sigma^{3/2} f\|_{L^2(\Omega)} \|\sigma^{-1/2} u\|_{L^2(\Omega)}.
 \end{equation}
 For the second term on the right-hand side of \eqref{eq:weighted_h1_reg}
 we get from \eqref{eq:weight_local}
 \begin{equation}\label{eq:weighted_h1_reg_2}
  \int_\Omega u\nabla u\cdot\nabla\sigma 
  = \sum_{i=1}^d \int_{\Omega_{\Gamma_i}} u_{\Gamma_i}(x_i,y_i)\nabla_{\Gamma_i} u_{\Gamma_i}(x_i,y_i)\cdot \begin{pmatrix}0\\1\end{pmatrix} \mathrm dx_i\mathrm dy_i\\
  = \sum_{i=1}^d \int_{\Omega_{\Gamma_i}} \frac12 \partial_{y_i} u_{\Gamma_i}(x_i,y_i)^2\mathrm dx_i\mathrm dy_i .
 \end{equation}
 Integration by parts and exploiting the fact that $u$ vanishes on $\Gamma$ yields
 \begin{align*}
 &\phantom{\le}\int_0^{\overline x_i}\int_0^{\overline y_i(x_i)}  \frac12 \partial_{y_i} u_{\Gamma_i}(x_i,y_i)^2\mathrm dy_i\mathrm dx_i 
 =\frac12 \int_0^{\overline x_i} u_{\Gamma_i}^2(x_i,\overline y_i(x_i))\mathrm dx_i \\
 &\ge \frac{c_*}2 \int_0^{\overline x_i} u_{\Gamma_i}^2(x_i,\overline y_i(x_i))\sqrt{1+\overline y_i'(x_i)^2}\mathrm dx_i 
 = \frac{c_*}2 \|u\|_{L^2(\partial\Omega_{\Gamma_i}\setminus\Gamma)}^2.
 \end{align*}
 The constant \[c_*:=\min_{i\in\mathcal C}\essinf_{x_i\in(0,\overline x_i)}1/\sqrt{1+\overline y'_i(x_i)^2}\] depends solely on the geometry of $\Omega$.
 Insertion into \eqref{eq:weighted_h1_reg_2} yields
 \begin{equation}\label{eq:weighted_h1_reg_3}
  \int_\Omega u\nabla u\cdot\nabla\sigma \ge c_* \|u\|_{L^2(\tilde\Gamma)}^2,\quad \tilde\Gamma:=\bigcup_{i,j=1}^d \partial\Omega_i\cap\partial\Omega_j.
 \end{equation}
 Combining the estimates \eqref{eq:weighted_h1_reg}, \eqref{eq:weighted_h1_reg_1} and \eqref{eq:weighted_h1_reg_3}
 leads with Young's inequality to
 \begin{equation}\label{eq:weighted_h1_reg_4}
  \|\sigma^{1/2}\nabla u\|_{L^2(\Omega)}^2 +  c_* \|u\|_{L^2(\tilde\Gamma)}^2 \le c |\ln d_I|^{2}\|\sigma^{3/2} f\|_{L^2(\Omega)}^2 + \varepsilon |\ln d_I|^{-2}\|\sigma^{-1/2} u\|_{L^2(\Omega)}^2.
 \end{equation} 
 It remains to appropriately bound the latter term in \eqref{eq:weighted_h1_reg_4} to show a weighted $L^2(\Omega)$-estimate for $u$. 
 The decomposition into the subsets $\Omega_i$, integration by parts and Young's inequality yield
 \begin{align}\label{eq:weighted_h1_reg_5}
  \|\sigma^{-1/2} u\|_{L^2(\Omega)}^2 &= \sum_{i=1}^d \int_0^{\overline x_i} \int_0^{\overline y_i(x_i)} \frac1{d_I+y_i} u_{\Gamma_i}(x_i,y_i)^2\mathrm dy_i\mathrm dx_i\nonumber\\
  &= \sum_{i=1}^d \Bigg(\int_0^{\overline x_i} \ln\left(d_I+\overline y_i(x_i)\right) u_{\Gamma_i}(x_i,\overline y_i(x_i))^2\mathrm dx_i\nonumber\\
  &\quad- \int_0^{\overline x_i} \int_0^{\overline y_i(x_i)} \ln\left(d_I+y_i\right) 2u_{\Gamma_i}(x_i,y_i) \partial_{y_i} u_{\Gamma_i}(x_i,y_i)\mathrm dy_i\mathrm dx_i\Bigg)\nonumber\\
  &\le c_{**}\|u\|_{L^2(\tilde\Gamma)}^2 + \hat c |\ln d_I|^2 \|\sigma^{1/2} \nabla u\|_{L^2(\Omega)}^2 + \frac12 \|\sigma^{-1/2} u\|_{L^2(\Omega)}^2,
 \end{align}
 with
 \[
   \hat c > 0\quad\mbox {and}\quad c_{**}:=2\max_{i\in\mathcal C} \esssup_{x_i\in(0,\overline x_i)}
   \frac{\ln\left(1+\bar y_i(x_i)\right)}{\sqrt{1+\overline y_i'(x_i)^2}}.
\]
 The latter term in \eqref{eq:weighted_h1_reg_5} may be kicked back to the left-hand side.
 Inserting \eqref{eq:weighted_h1_reg_5} into \eqref{eq:weighted_h1_reg_4} and choosing 
 \[\varepsilon = \frac14 \min\left\lbrace \frac{c_*}{c_{**}},\frac1{\hat c}\right\rbrace\]
 yields
 \begin{align}\label{eq:weighted_h1_reg_6}
 & \|\sigma^{1/2}\nabla u\|_{L^2(\Omega)}^2 +  c_* \|u\|_{L^2(\tilde\Gamma)}^2 \nonumber\\
 &\qquad  \le c |\ln d_I|^2\|\sigma^{3/2} f\|_{L^2(\Omega)}^2 + \frac12 \left(\|\sigma^{1/2}\nabla u\|_{L^2(\Omega)}^2 +  c_* |\ln d_I|^{-2}\|u\|_{L^2(\tilde\Gamma)}^2 \right).
 \end{align}
 Due to $d_I<\e^{-1}$, a kick-back argument leads to the desired estimate for the term $\|\sigma^{1/2} \nabla u\|_{L^2(\Omega)}$.
 Using the estimates \eqref{eq:weighted_h1_reg_5} and \eqref{eq:weighted_h1_reg_6} we finally confirm that
 \[
  \|\sigma^{-1/2}u\|_{L^2(\Omega)}^2 \le c \left(\|u\|_{L^2(\tilde\Gamma)}^2 + |\ln d_I|^2\|\sigma^{1/2} \nabla u\|_{L^2(\Omega)}^2 \right)
  \le c |\ln d_I|^4\|\sigma^{3/2} f\|_{L^2(\Omega)}^2.
 \]
 \end{proof}
 For technical reasons we decompose our domain into dyadic subsets 
 \begin{equation}\label{eq:dyadic}
   \Omega_J:=\{x\in\Omega\colon \rho(x)\in (d_{J+1},d_J)\},\quad J=0,\ldots,I,
 \end{equation}
 where given $d_I\in(0,\e^{-1})$, the numbers $d_J$ fulfill $d_{I+1}=0$, $d_J=2d_{J+1}$ for $J=I-1,\ldots,1$, and $d_0 =\diam(\Omega)$. Without loss of generality we assume that $\left|\Omega_0\right|\neq 0$, otherwise a simple scaling argument can be used to achieve this. By means of the subsets $\Omega_J$, we will be able to handle the weight function $\sigma$ within the proofs. More precisely, we will especially use that
 \begin{equation}\label{eq:relation_sigma_dJ}
	\inf_{x\in\Omega_J}\sigma(x)\sim\sup_{x\in\Omega_J}\sigma(x) \sim d_J,\quad J=0,\ldots,I.
 \end{equation}
 In the next lemma we show a local regularity result on the subsets $\Omega_J$ and, as a consequence of this, global a priori estimates for second derivatives in weighted norms.
 Due to pollution effects we have to take into account the patches defined by
 \[
  \Omega_J':= \Omega_{J-1}\cup\Omega_J\cup\Omega_{J+1},\qquad \Omega_J'':=\Omega_{J-1}'\cup\Omega_J'\cup\Omega_{J+1}'
 \]
  with the obvious modifications for the cases $J=I, I-1$, and $J=0,1$.
 \begin{lemma}\label{lem:weighted_reg_h2}  
  Let $\Omega\subset\mathbb R^2$ be a polygonal domain.
  The solution $u$ of \eqref{eq:weak_form_1} satisfies
  \begin{align*}
   &(i)\quad & \|\nabla^2 u\|_{L^2(\Omega_J)} &\le c \left(d_J^{-1}\|\nabla
     u\|_{L^2(\Omega_J')} + \|f\|_{L^2(\Omega_J')}\right),\quad 
   J=0,\ldots,I,
 \intertext{provided that $u\in H^2(\Omega_J')$. Moreover, if $\Omega$ is convex, the estimates}
   &(ii)\quad & \|\sigma\nabla^2 u\|_{L^2(\Omega)} &\le c\|\sigma f\|_{L^2(\Omega)},\\
   &(iii)\quad & \|\sigma^{3/2}\nabla^2 u\|_{L^2(\Omega)} &\le c \left|\ln d_I\right|
   \|\sigma^{3/2}f\|_{L^2(\Omega)}
  \end{align*}
  are fulfilled.
 \end{lemma}
 \begin{proof} 
  We consider a covering of $\Omega_J$ consisting of finitely many balls $B_{d_J/8}(x_i)$, $i=1,\ldots,N$, with radius $d_J/8$ and centers $x_i\in\Omega_J$. 
  This implies $B_{d_J/4}(x_i)\subset \Omega_J'$. In a first step, we appropriately bound $\left|\nabla^2 u\right|$ on $B_{d_J/8}(x_i)$.
  For this purpose, we introduce a smooth cut-off function $\eta\in C_0^\infty(\Omega)$ 
  satisfying $\eta\equiv 1$ in $B_{d_J/8}(x_i)$ and $\supp \eta \subset B_{d_J/4}(x_i)$.  
  In case of $J=I$ the balls may overlap the boundary. Then, $B_{d_J/8}(x_i)$ and $B_{d_J/4}(x_i)$ are the intersections of each ball with $\Omega$.
  It is possible to construct $\eta$ in such a way that $\|D^\alpha \eta\|_{L^\infty(\Omega)}\le cd_J^{-|\alpha|}$. Next, let $\bar u \in \mathbb{R}$ be defined by
  \[
  	\bar u :=\begin{cases}
  				|B_{d_J/4}(x_i)|^{-1}\int_{B_{d_J/4}(x_i)} u&\text{if }\supp \eta \subset \Omega,\\
  				0 &   \text{otherwise.}
  			\end{cases}
  \]
  The function $\eta (u-\bar u)\in H^1_0(\Omega)$ is the weak solution of the boundary value problem
  \[
   -\Delta(\eta (u-\bar u)) = -\Delta \eta (u-\bar u) - 2\nabla u\cdot\nabla \eta + \eta f \quad\mbox{in }\Omega,\qquad \eta (u-\bar u) = 0 \quad\mbox{on }\Gamma.
  \]
  Note that the right-hand side belongs to $L^2(\Omega)$.
  With standard regularity results \cite{Gri85} we conclude
 \begin{align*}
  \|\nabla^2 u\|_{L^2(B_{d_J/8}(x_i))}  
  &\le c\|\nabla^2(\eta (u-\bar u))\|_{L^2(\Omega)}
  \le c \|\Delta(\eta(u-\bar u))\|_{L^2(\Omega)} \\
  &\le c\left(d_J^{-2}\|u-\bar u\|_{L^2(B_{d_J/4}(x_i))} + d_J^{-1}\|\nabla u\|_{L^2(B_{d_J/4}(x_i))} + \|f\|_{L^2(B_{d_J/4}(x_i))}\right).
 \end{align*} 
 An application of the Poincar\'e inequality allows to bound the first term on the right-hand side by the second one. 
 Note that a careful choice of the midpoints $\{x_i\} \subset \Omega_J$ according to \cite[Lemma A.1]{MW12} guarantees 
 that in each point $x\in \Omega_J'$ only a finite number $n$ of balls $B_{d/4}(x_i)$ overlap. The number $n$ depends only on the spatial dimension.
 Hence, we get the desired estimate (i), 
 \begin{align*}
  \|\nabla^2 u\|_{L^2(\Omega_J)}^2 &\le \sum_{i=1}^N\|\nabla^2 u\|_{L^2(B_{d_J/8}(x_i))}^2 
  \le c\sum_{i=1}^N \left(d_J^{-2}\|\nabla u\|_{L^2(B_{d_J/4}(x_i))}^2 + \|f\|_{L^2(B_{d_J/4}(x_i))}^2\right) \\
  &\le c \left(d_J^{-2}\|\nabla u\|_{L^2(\Omega_J')}^2  + \|f\|_{L^2(\Omega_J')}^2\right).
 \end{align*}
 The estimate (ii) follows from (i)
 taking into account the relation \eqref{eq:relation_sigma_dJ}. From this we deduce
 \begin{align*}
  \|\sigma\nabla^2 u\|_{L^2(\Omega)}^2
  &\le c\sum_{J=0}^I d_J^2 \|\nabla^2 u\|_{L^2(\Omega_J)}^2 
  \le c\sum_{J=0}^I \left(\|\nabla u\|_{L^2(\Omega_J')}^2 + d_J^2\|f\|_{L^2(\Omega_J')}^2\right) \\
  &\le c \left(\|\nabla u\|_{L^2(\Omega)}^2 + \|\sigma f\|_{L^2(\Omega)}^2\right).
 \end{align*}
 With Lemma \ref{lem:weighted_reg_h1} we conclude the assertion.
 In the same way the estimate (iii) follows.
 \end{proof}
 
 As we consider computational domains having a polygonal boundary, we have to deal with singularities occurring in the vicinity of 
 vertices of the domain as well. For an accurate description of these singularities, we exploit regularity results in weighted Sobolev spaces with weights related to the corners. 
 We denote the distance functions to the corners $\boldsymbol c_j$ by $r_j(x):=|x-\boldsymbol c_j|$,
 $j\in\mathcal C$.
 Moreover, we introduce the regions $\Omega^j_R:=\{x\in\Omega\colon r_j(x) < R\}$, and choose $R>0$ appropriately such that these domains do not intersect. 
 Furthermore, we introduce the region $\hat \Omega_R:=\Omega\setminus\cup_{j\in\C} \Omega_{R}^j$.
 On each $\Omega_R^j$ we define for $k\in\mathbb N_0$, $p\in [1,\infty]$ and $\beta_j\in\mathbb R$ the local norm
 \begin{align*}
  \|v\|_{V^{k,p}_{\beta_j}(\Omega_R^j)}^p := \sum_{|\alpha|\le k} \|r_j^{\beta_j+|\alpha|-k} D^\alpha v\|_{L^p(\Omega_R^j)}^p, &\quad \mbox{for}\ p\in[1,\infty),\\
  \|v\|_{V^{k,\infty}_{\beta_j}(\Omega_R^j)} := \max_{|\alpha|\le k} \|r_j^{\beta_j+|\alpha|-k} D^\alpha v\|_{L^\infty(\Omega_R^j)}. &  
 \end{align*}
 The weighted Sobolev space $V_{\vec\beta}^{k,p}(\Omega)$ with weight vector $\vec\beta\in\mathbb R^d$ 
 is defined as the set of measurable functions  with finite norm
 \begin{equation}\label{eq:weighted_norm}
  \|v\|_{V^{k,p}_{\vec\beta}(\Omega)} :=  \|v\|_{W^{k,p}(\hat \Omega_{R/2})} + \sum_{j=1}^d \|v\|_{V^{k,p}_{\beta_j}(\Omega_R^j)}.
 \end{equation}
 We will frequently use these norms on subdomains $\mathcal G\subset \Omega$. In this case, the weight functions $r_j$ are still related to the corners 
 of $\Omega$ and not of $\mathcal G$.
 
 Under certain assumptions on the input data, one can show that the solution of \eqref{eq:weak_form_1} belongs to these weighted Sobolev space 
 provided that the weights are sufficiently large. The lower bounds for the
 weights depend on the singular exponents 
 \[
 	\lambda_j:=\pi/\omega_j,\ j\in\mathcal C.
 \]
 The following result is taken from \cite[\S1.3, Theorem 3.1]{NP94} for $p=2$, and \cite[Theorem 2.6.1]{KMR01} for $p\in(1,\infty)$. 
 \begin{lemma}\label{lem:V2p_reg}
   Let $f\in V^{0,p}_{\vec\beta}(\Omega)$ with $p\in(1,\infty)$, and $\vec\beta\in \mathbb R^d$ satisfying $\beta_j\in (2-2/p-\lambda_j, 2-2/p)$ for all $j\in\mathcal C$.
   Then, the solution $u$ of \eqref{eq:weak_form_1} belongs to $V^{2,p}_{\vec\beta}(\Omega)$
   and satisfies
   \[
     \|u\|_{V^{2,p}_{\vec\beta}(\Omega)} \le c \|f\|_{V^{0,p}_{\vec\beta}(\Omega)}.
   \]
 \end{lemma}
 \begin{remark}\label{remark:normalderivative}
 	According to \cite[Theorem 7.1.1]{KMR97} (see also \cite[Lemma 2.32]{Pfe}) there holds
 	\[
 		\sum_{j=1}^d\left(\sum_{\left|\alpha\right|\leq 1}\left|(D^\alpha v)(\boldsymbol c_j)\right|\right)=0
 	\]
 	if $v\in V^{2,2}_{\vec\beta}(\Omega)$ with $\beta_j<0$ for $j\in\mathcal{C}$. Thus, if $f\in V^{0,2}_{\vec\beta}(\Omega)$ with $\beta_j<0$ for $j\in\mathcal{C}$, then the normal derivative $\partial_n u$ is equal to zero at each convex corner. At non-convex corners it has a pole in general, see also the discussions in \cite{AMPR15}. 
 \end{remark}
 In order to derive optimal error estimates, we need a similar result for the case $p=\infty$, which is excluded in the previous lemma.
 However, taking regularity results in weighted H\"older spaces into account (see \cite{KMR01}) the assertion of Lemma \ref{lem:V2p_reg} 
 remains true when assuming slightly more regularity for the right-hand side. 
 For the proof of the following result we refer to \cite[Lemma 4.2]{Sir10}.
 \begin{lemma}\label{lem:V2inf_reg}
  Assume that $f\in C^{0,\sigma}(\overline\Omega)$ with some $\sigma\in (0,1)$.
  Let the weight vector $\vec\beta\in[0,2)^d$ be chosen such that $\beta_j>2-\lambda_j$ for all $j\in\C$.
  Then, the solution of \eqref{eq:weak_form_1} belongs to $V^{2,\infty}_{\vec\beta}(\Omega)$ and satisfies the \emph{a priori} estimate
  \[
   \|u\|_{V^{2,\infty}_{\vec\beta}(\Omega)} \le c \|f\|_{C^{0,\sigma}(\overline\Omega)}.
  \] 
 \end{lemma}

\section{Weighted $L^2(\Omega)$ error estimates}\label{sec:weighted_estimate}
We approximate the solution of \eqref{eq:weak_form_1} with linear finite elements. Therefore we introduce a 
family of conforming triangulations $\{\mathcal T_h\}_{h>0}$ consisting of triangular elements,
where $h:=\max_{T\in\mathcal T_h}\diam(T)$ denotes the mesh parameter.
As specialty, we consider triangulations which are isotropically refined towards the whole boundary:
Let $\rho_T:=\dist(T,\Gamma)$ the distance of the element $T\in\mathcal
T_h$ to the boundary $\Gamma$.
We assume that
\begin{equation}\label{eq:ref_cond}
h_T:=\diam(T)\sim \begin{cases}
h^2, &\mbox{if}\ \rho_T=0,\\
h \sqrt{\rho_T}, &\mbox{if}\ \rho_T>0,
\end{cases}\qquad \forall T\in\mathcal T_h.
\end{equation}
This refinement condition ensures that elements touching the boundary have diameter $h^2$. Moreover, elements with $O(1)$-distance to the boundary have diameter $h$, and adjacent elements have approximately equal diameter.
\begin{remark}
	While the number of elements for quasi-uniform triangulations of planar domains behaves
	like $h^{-2}$, there is a slight increase in the number of elements when the refinement condition \eqref{eq:ref_cond}
	holds. Let $S_h:=\cup\{T\in\mathcal T_h\colon \rho_T=0\}$. 
	The number of elements belonging to the set $S_h$, can be estimated by
	\[
	N_{\text{elem}}^{\text{bd}} \sim \frac{|\Gamma|}{\min_{T\subset S_h} |T|} \sim h^{-2}.
	\]
	However, for the number of elements $N_{elem}^{\text{int}}$ away from the boundary there holds
	\begin{align*}
	N_{\text{elem}}^{\text{int}} = \sum_{T\not{\subset}S_h} 1 = \sum_{T\not\subset S_h} |T|^{-1} \int_{T}\mathrm dx 
	\sim h^{-2} \int_{\Omega\setminus S_h} \rho(x)^{-1}\mathrm dx \sim h^{-2}|\ln h|,
	\end{align*}
	where we exploited $h_T\sim h \sqrt{\rho_T}$ and the property $\rho_T \sim \rho(x)$ for all $x\in T$ in case of $\rho_T>0$.
\end{remark}

Now, we define the finite-dimensional space $V_{0h}:=V_h\cap H^1_0(\Omega)$ with
\[
V_h:=\{v_h\in C(\overline\Omega)\colon v_h|_T\in\mathcal P_1(T)\ \mbox{for all}\ T\in\mathcal T_h\},
\]
where $\mathcal P_1(T)$ denotes the set of polynomials on the element $T$ of degree at most $1$, and determine approximations to $u$ by solving the problem:
\begin{equation}\label{eq:fe_formulation}
\mbox{Find}\ u_h\in V_{0h}\colon\quad  (\nabla u_h,\nabla v_h)_{L^2(\Omega)} = (f,v_h)_{L^2(\Omega)} \quad \mbox{for all}\ v_h\in V_{0h}.
\end{equation}
The aim of this section is to derive an error estimate in a weighted $L^2(\Omega)$-norm. Such a term occurs in the applications we have in mind, and will become clear in Section \ref{sec:surfaceflux}. More precisely, the term $\|\sigma^{-3/2}(u-u_h)\|_{L^2(\Omega)}$ with $\sigma=\rho + d_I$ from \eqref{eq:weight} 
is considered, where the number $d_I$ satisfies $d_I=2^{-I}$. The exponent $I$ is chosen such that 
$d_I=c_I h^2$ with some fixed and mesh-independent constant $c_I>1$,
which we specify later. This construction implies $I\sim \lnh$.

As the mesh size solely depends on the distance to the boundary which is bounded within $\Omega_J$ by  $d_J$ and $d_{J+1}=d_J/2$,
the meshes are locally quasi-uniform within each $\Omega_J$, $J=0,\ldots,I$.
That means, there are constants $c_1,c_2>0$ such 
that each $T\in\mathcal T_h$ with $T\cap \Omega_{J}\ne\emptyset$ satisfies
\begin{equation}\label{eq:mesh_size_dJ}
\begin{aligned}
c_1 h \sqrt{d_J} &\le h_T \le c_2 h \sqrt{d_J} && \mbox{if}\quad J=0,\ldots,I-1,\\
c_1 c_I^{-1}h\sqrt{d_I}&\le h_T \le c_2 h \sqrt{d_I}&&\mbox{if}\quad J=I.
\end{aligned}
\end{equation}

We are now in the position to derive the main result of this section under the assumption that the computational domain is convex.
The non-convex case will be discussed later as different assumptions and techniques will be used.
\begin{theorem}\label{thm:weighted_estimate}
 Let $\Omega\subset\mathbb R^2$ be a convex polygonal domain, this is, $\overline\lambda:=\min_{j\in\C}\lambda_j > 1$.
 Assume that $f\in C^{0,\sigma}(\overline\Omega)$ with some $\sigma\in (0,1)$.
 For $c_I>1$ sufficiently large, there exists some $h_0=h_0(c_I)>0$  such that the estimate
 \begin{equation}\label{eq:weighted_estimate}
   \|\sigma^{-3/2}(u-u_h)\|_{L^2(\Omega)} \le c h^{\min\{2,-1+2\bar\lambda-2\varepsilon\}}\lnh^{3/2}
   \|f\|_{C^{0,\sigma}(\overline \Omega)}
 \end{equation}
 holds for all $h\le h_0$ 
 and $\varepsilon>0$.  
\end{theorem}
\begin{proof} 
 The norm on the left-hand side of \eqref{eq:weighted_estimate} possesses the representation
 \[
  \|\sigma^{-3/2}(u-u_h)\|_{L^2(\Omega)} = \sup_{\genfrac{}{}{0pt}{}{\varphi\in L^2(\Omega)}{\|\varphi\|_{L^2(\Omega)}=1}} \left(u-u_h,\sigma^{-3/2}\varphi\right)_{L^2(\Omega)}.
 \]
 Let $w\in H^1_0(\Omega)$ be the solution of the dual problem
 \begin{equation}\label{eq:dual}
  -\Delta w = \sigma^{-3/2}\varphi \quad\mbox{in } \Omega,\qquad w=0\quad\mbox{on } \Gamma.
 \end{equation}
 Then, we obtain using Galerkin orthogonality and the Cauchy-Schwarz inequality
 \begin{align}\label{eq:nitsche}
  \|\sigma^{-3/2}(u-u_h)\|_{L^2(\Omega)} &= (\nabla(u-u_h),\nabla(w-I_h w))_{L^2(\Omega)} \nonumber\\
  &\le \sum_{J=0}^I\|u-u_h\|_{H^1(\Omega_J)}\|w-I_h w\|_{H^1(\Omega_J)},
 \end{align} 
 where $I_h w$ denotes the Lagrange interpolant of $w$.
 An application of the local finite element error estimate from \cite[Theorem 3.4]{DGS11}, 
 the interpolation error estimate 
 \[
  \|u-I_h u\|_{H^\ell(T)} \le c h_T^{2-\ell} \|\nabla^2 u\|_{L^2(T)} \le c h^{2-\ell} d_J^{(2-\ell)/2}\|\nabla^2 u\|_{L^2(T)},\quad \ell=0,1,
 \] 
 and $h d_J^{-1/2} \le c_I^{-1/2} \le c$  yield the estimate
 \begin{align}
   \|u-u_h\|_{H^1(\Omega_{J})}  &\le c \left(\inf_{\genfrac{}{}{0pt}{}{\chi\in V_h}{\chi|_\Gamma = u_h|_\Gamma}} \left(\|\nabla(u-\chi)\|_{L^2(\Omega_{J}')}
  + d_J^{-1} \|u-\chi\|_{L^2(\Omega_{J}')}\right) + d_J^{-1}\|u-u_h\|_{L^2(\Omega_{J}')}\right) \label{eq:local_h1}\\
  &\le c\left(h d_J^{1/2} \|\nabla^2 u\|_{L^2(\Omega_J'')} + d_J^{-1}\|u-u_h\|_{L^2(\Omega_J')}\right)\label{eq:local_h1_2}
 \end{align}
 for all $J=0,\ldots,I$.
 For the dual solution we get in an analogous way the estimate
 \begin{equation}\label{eq:int_error_w}
  \|\nabla(w-I_h w)\|_{L^2(\Omega_J)} \le c h d_J^{1/2}\|\nabla^2 w\|_{L^2(\Omega_J')}
 \end{equation}
 for all $J=0,\ldots,I$. 
 Insertion of \eqref{eq:local_h1_2} and \eqref{eq:int_error_w} into \eqref{eq:nitsche},
 and summation over all subsets while taking into account \eqref{eq:relation_sigma_dJ} leads to
 \begin{align*}
  &\|\sigma^{-3/2}(u-u_h)\|_{L^2(\Omega)}\\
   &\quad\le c \Big(h^2 \|\sigma^{-1/2}\nabla^2 u\|_{L^2(\Omega)} \|\sigma^{3/2}\nabla^2 w\|_{L^2(\Omega)} +h\|\sigma^{-3/2}(u-u_h)\|_{L^2(\Omega)} \|\sigma \nabla^2 w\|_{L^2(\Omega)}\Big) \\
  &\quad \le c \Big(h^2 \lnh \|\sigma^{-1/2}\nabla^2 u\|_{L^2(\Omega)} \|\varphi\|_{L^2(\Omega)} 
  + h \|\sigma^{-3/2}(u-u_h)\|_{L^2(\Omega)}\|\sigma^{-1/2}\varphi\|_{L^2(\Omega)}\Big).
 \end{align*}
 In the last step we applied the a priori estimates from Lemma \ref{lem:weighted_reg_h2}.
 Next, we exploit the property $\sigma^{-1/2} \le d_I^{-1/2}$ and the assumption $\|\varphi\|_{L^2(\Omega)} =1$.
 Moreover, by choosing $c_I$ sufficiently large, we obtain due to the relation $d_I=c_I h^2$,
 \[
  c h d_I^{-1/2} = c c_I^{-1/2} \le 1/2.
 \]
 Hence, we can kick back the latter term to the left-hand side. This finally implies
 \begin{equation}\label{eq:last_step}
  \|\sigma^{-3/2}(u-u_h)\|_{L^2(\Omega)} \le 
  c h^2 \lnh \|\sigma^{-1/2}\nabla^2 u\|_{L^2(\Omega)}.
 \end{equation} 
 It remains to estimate the weighted norm on the right-hand side.
 Therefore, we use the decomposition $\Omega=(\cup_{j\in \mathcal C} \Omega_R^j)\cup \hat \Omega_{R/2}$ 
 already used in the norm definition \eqref{eq:weighted_norm}.
 In the interior of the domain we bound the norm on the right-hand side of \eqref{eq:last_step} by
 \[\label{eq:hoelder_interior}
  \|\sigma^{-1/2} \nabla^2 u\|_{L^2(\hat\Omega_{R/2})} \le \|\sigma^{-1/2}\|_{L^2(\hat\Omega_{R/2})} \|\nabla^2 u\|_{L^\infty(\hat\Omega_{R/2})}.
 \]
 In each subset $\Omega_R^j$ we apply the estimate
 \[\label{eq:hoelder_corner}
  \|\sigma^{-1/2}\nabla^2 u\|_{L^2(\Omega_R^j)}
  \le \|\sigma^{-1/2} r_j^{-\beta_j}\|_{L^2(\Omega_R^j)} \|r_j^{\beta_j} \nabla^2 u\|_{L^\infty(\Omega_R^j)}.
 \]
 The previous inequalities and the regularity results from Lemma \ref{lem:V2inf_reg} imply
 \begin{align}
	 \|\sigma^{-1/2}\nabla^2 u\|_{L^2(\Omega)} &\le 
	 c \left(\|\sigma^{-1/2}\|_{L^2(\hat\Omega_{R/2})}+\max_{j\in\mathcal{C}}\|\sigma^{-1/2} r_j^{-\beta_j}\|_{L^2(\Omega_R^j)}\right)\|u\|_{V^{2,\infty}_{\vec\beta}(\Omega)}\notag\\
	 &\le 
	 c \left(\|\sigma^{-1/2}\|_{L^2(\hat\Omega_{R/2})}+\max_{j\in\mathcal{C}}\|\sigma^{-1/2} r_j^{-\beta_j}\|_{L^2(\Omega_R^j)}\right)\|f\|_{C^{0,\sigma}(\overline\Omega)},\label{eq:weightedreg}
 \end{align}
 provided that $\beta_j=\max\{0,2-\lambda_j+\varepsilon\}<2$ with $\varepsilon>0$. Once we have shown that
 \begin{equation}\label{eq:estimatesigma}
	 \|\sigma^{-1/2}\|_{L^2(\hat\Omega_{R/2})}+\max_{j\in\mathcal{C}}\|\sigma^{-1/2} r_j^{-\beta_j}\|_{L^2(\Omega_R^j)}\le c |\ln h|^{1/2} h^{\min\{0,1-2\overline\beta\}}
 \end{equation}
 with $\overline\beta:=\max\{0,2-\overline\lambda+\varepsilon\}$ and $\varepsilon>0$ sufficiently small the assertion follows. The proof of \eqref{eq:estimatesigma} is postponed to Lemma \ref{lem:integrals}.
 \end{proof}
 
 \begin{lemma}\label{lem:integrals}
  Let $\Omega$ be convex. For $\beta_j\in [0,1)$, $j\in\mathcal C$, there are the estimates 
  \begin{align}
   \|\sigma^{-1/2} r_j^{-\beta_j} \|_{L^2(\Omega_R^j)} &\le c \lnh^{1/2} \times \begin{cases} h^{\min\{0,1-2\beta_j\}} & \text{if }\beta_j\neq \frac12,\\ \lnh^{1/2} &\text{if }\beta_j=\frac12,\end{cases}\label{eq:int1}\\
   \|\sigma^{-1/2}\|_{L^2(\hat\Omega_{R/2})} &\le c \lnh^{1/2}.
  \end{align}
 \end{lemma}
 \begin{proof}
  Recall the decomposition of $\Omega$ already used in the proof of Lemma \ref{lem:weighted_reg_h1}.
  There, we introduced domains $\Omega_{\Gamma_j}$, $j\in\mathcal C$, such that $\dist(x,\Gamma_j) = \rho(x)$ for all $x\in\Omega_{\Gamma_j}$.
  Moreover, we constructed integration bounds in the local coordinates $(x_j,y_j)$, i.e., $0<x_j < \bar x_j$ and $0<y_j < \bar y_j(x_j)$.
  Based on this, we first show \eqref{eq:int1}. Due to symmetry reasons it suffices to estimate the integral on the subset $\Omega_{\Gamma_j}\cap\Omega_R^j$.
  This is done in two steps according to the coloring in Figure \ref{fig:integration1}.
  \begin{figure}[htb]
    \begin{center}
      \subfloat[Decomposition of $\Omega_{\Gamma_j} \cap \Omega_R^j$\label{fig:integration1}]{
        \includegraphics[height=4.5cm]{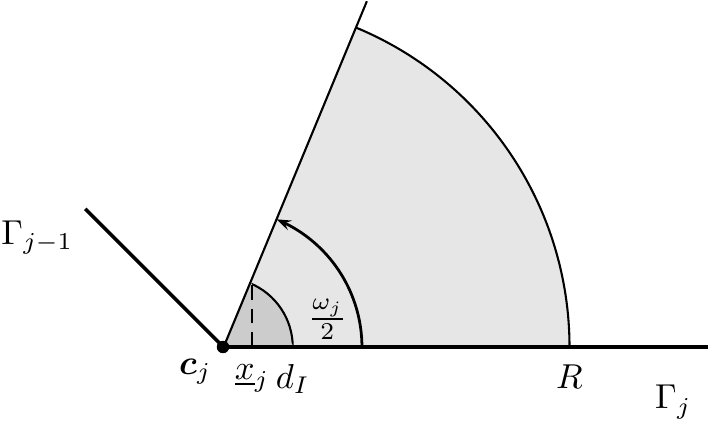}
      }\qquad
      \subfloat[Overlap of $\Omega_{\Gamma_j}\cap \hat\Omega_{R/2}$\label{fig:integration2}]{
        \includegraphics[height=4cm]{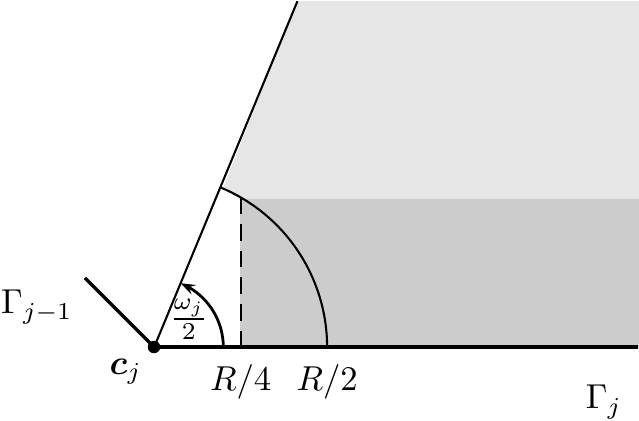}
      }      
    \end{center}
    \caption{Illustration of the integration domains used in the proof of Lemma \ref{lem:integrals}.}
    \label{fig:integration}
  \end{figure}
  First,
  in the circular sector $\Omega_{\Gamma_j}\cap B_{d_I}$, where $B_{d_I}$ denotes the ball with radius $d_I=c_Ih^2$ around the corner $\boldsymbol c_j$ (see also the dark gray region in Figure \ref{fig:integration1}),
  we use polar coordinates and obtain
  \begin{equation}\label{eq:integral_1}
   \|\sigma^{-1/2} r_j^{-\beta_j}\|_{L^2(\Omega_{\Gamma_j}\cap B_{d_I})}^2
   \le d_I^{-1} \int_0^{\omega_j/2} \int_0^{d_I} r_j^{1-2\beta_j}\mathrm dr_j\mathrm d\varphi_j
   \le c \lnh h^{2\min\{0,1-2\beta_j\}}.
  \end{equation} 
 The remaining subdomain $(\Omega_R^j\cap\Omega_{\Gamma_j})\backslash B_{d_I}$ (illustrated by the light gray region in Figure~\ref{fig:integration1}) is enlarged to the rectangular domain bounded by 
 $\underline x_j < x_j < R$ with $\underline x_j = \sin(\omega_j/2) d_I \sim d_I$ and $0 < y_j < R$. Moreover, we exploit that $r_j > x_j > 0$. Having in mind that $d_I=c_Ih^2$, this leads~to
 \begin{equation}\label{eq:integral_2}
  \|\sigma^{-1/2} r_j^{-\beta_j}\|_{L^2(\Omega_R^j\cap\Omega_{\Gamma_j}\setminus B_{d_I})}^2
  \le \int_{\underline x_j}^{R} \int_0^{R} \frac{x_j^{-2\beta_j}}{d_I+y_j}\mathrm dy_j \mathrm dx_j \le c \lnh\times \begin{cases}
  h^{2\min\{0,1-2\beta_j\}} & \text{if }\beta_j\neq \frac12,\\ \lnh &\text{if }\beta_j=\frac12.
  \end{cases}
 \end{equation}
 The inequalities \eqref{eq:integral_1} and \eqref{eq:integral_2}, together with
 analogous arguments for the domain $\Omega_R^j\cap\Omega_{\Gamma_{j-1}}$, imply the desired estimate on $\Omega_R^j$.
 
 To show the estimate on $\hat\Omega_{R/2}$ we first integrate over the domains
 \[\breve\Omega_{R/4}^j:=\{F_i(x_j,y_j)\in\mathbb R^2\colon R/4 < x_j <
   \bar x_j-R/4,\ 0<y_j < \sqrt3R/4\},\]
 see also the dark gray region in Figure \ref{fig:integration2}.
 For each $j\in\mathcal C$ we obtain the estimate
 \[
  \|\sigma^{-1/2}\|_{L^2(\breve\Omega_{R/4}^j)}^2 \le c \int_{R/4}^{\overline x_j-R/4} \int_0^{\sqrt{3}R/4} (d_I+y_j)^{-1}\mathrm dy_i\mathrm dx_i
  \le c \lnh.
 \]
 On the remaining set $\hat\Omega_{R/2}\setminus \cup_{j=1}^d \breve\Omega_{R/4}^j$ the weight $\sigma$ is of order one and vanishes in the generic constant.
 Thus,
 \[
  \|\sigma^{-1/2}\|_{L^2(\hat\Omega_{R/2}\setminus\cup_{j=1}^d\breve\Omega_{R/4}^j)}^2 \le c.
 \]
 This implies the second estimate.
 \end{proof}
 
 In the remainder of this section we prove an analogue of Theorem \ref{thm:weighted_estimate} which not only requires less regular data but also holds in non-convex domains.
 This requires indeed some rigorous modifications as the solution of the dual problem 
 \eqref{eq:dual} fails to be in $H^2(\Omega)$ if $\Omega$ is non-convex. Moreover, we do not exploit weighted $W^{2,\infty}$-regularity of the solution, but remain in the 
 weighted $H^2$-setting.
 \begin{theorem}\label{thm:weighted_estimate_nonconvex}
 Let $\overline\lambda :=\min_{j\in\C}\lambda_j$.
 Assume that $f\in V^{0,2}_{\vec\alpha}(\Omega)$ with 
 $\alpha_j := \max\{0,1-\lambda_j+\varepsilon\}$ for all $j\in\mathcal C$ with arbitrary but sufficiently small $\varepsilon>0$. 
 For $c_I>1$ sufficiently large, there exists some $h_0=h_0(c_I)>0$ such that the estimate
 \begin{equation}\label{eq:weighted_estimate_nonconvex}
  \|\sigma^{-3/2}(u-u_h)\|_{L^2(\Omega)} \le c h^{\min\{1,-1+2\bar\lambda-2\varepsilon\}} \|f\|_{V^{0,2}_{\vec\alpha}(\Omega)}
 \end{equation}
 holds for all $h\le h_0$.
\end{theorem}
\begin{proof}  
  Having in mind the result of Theorem \ref{thm:weighted_estimate}, we first observe that we expect at most the convergence rate one in non-convex domains. Thus, we trade $\sigma^{-1/2}$ by $h^{-1}$ and it remains to show an estimate 
  in a weighted norm with a larger weight exponent. These ideas lead to
  \begin{equation}\label{eq:mod_nonconvex}
   \|\sigma^{-3/2}(u-u_h)\|_{L^2(\Omega)} \le c h^{-1} \|\sigma^{-1}(u-u_h)\|_{L^2(\Omega)}
   = c h^{-1}\sup_{\genfrac{}{}{0pt}{}{\varphi\in L^2(\Omega)}{\|\varphi\|_{L^2(\Omega)} = 1}} (u-u_h, \sigma^{-1}\varphi)_{L^2(\Omega)}.
 \end{equation}
  The dual problem reads in this case
  \begin{equation}\label{eq:dual_nonconvex}
   -\Delta w = \sigma^{-1}\varphi\quad \mbox{in}\ \Omega,\qquad w=0\quad \mbox{on}\ \Gamma.
  \end{equation}
  Analogous to \eqref{eq:nitsche}  we can show that   
  \begin{align}\label{eq:splitting_nonconvex}
   \|\sigma^{-1}(u-u_h)\|_{L^2(\Omega)} &\le c \sum_{J=0}^I \|u-u_h\|_{H^1(\Omega_J)} \|w-I_h w\|_{H^1(\Omega_J)},
  \end{align}
  and it remains to bound the local error terms on the right-hand side.
  In contrast to the proof of Theorem \ref{thm:weighted_estimate}, we have to distinguish between the inner subsets $\Omega_J$, $J=0,\ldots,I-3$ and the outer ones $J=I-2,I-1,I$, 
  as the dual solution $w$ may fail to be in $H^2(\Omega)$ due to the possibly non-convex corners.
  However, due to Lemma \ref{lem:V2p_reg}, the function $u$ belongs to
  $V^{2,2}_{\vec\alpha}(\Omega)$, which we are going to employ.
  
  In case that $J=0,\ldots,I-3$, we proceed as in the proof of Theorem \ref{thm:weighted_estimate}. Indeed, employing 
  the local estimate \eqref{eq:local_h1}, we obtain together with
  the property $r_{j,T}:=\dist(T,\boldsymbol c_j) \ge c d_J$, which holds for elements $T$ with $T\cap\Omega_J\ne\emptyset$,
  \begin{align}\label{eq:int_error_nonconvex}
   \|u-u_h\|_{H^1(\Omega_J)} &\le c \left(h d_J^{1/2}\|\nabla^2 u\|_{L^2(\Omega_J'')} + d_J^{-1}\|u-u_h\|_{L^2(\Omega_J')}\right) \nonumber\\
   &\le c \left(h d_J^{1/2-\overline\alpha} |u|_{V^{2,2}_{\vec\alpha}(\Omega_J'')} + d_J^{-1}\|u-u_h\|_{L^2(\Omega_J')}\right),
  \end{align}
  where we set $\overline \alpha =\max_{j\in \mathcal{C}}\alpha_j$. It is straightforward to confirm that the same estimate holds in the case $J=I-2,I-1,I$ as well.
  The only difference is, that local interpolation error estimates exploiting 
  weighted regularity, see e.\,g.\ \cite[Section 3.3]{ASW96}, have to be applied.
  Together with the refinement condition \eqref{eq:mesh_size_dJ} this leads to
  \[
    \|u-I_h u\|_{H^\ell(T)} \le c
    \begin{cases}
      h^{2(2-\ell-\alpha_j)} |u|_{V^{2,2}_{\alpha_j}(T)}, &\mbox{if}\ r_{j,T} = 0,\\
      h^{2-\ell}d_I^{1-\ell/2} r_{j,T}^{-\alpha_j} |u|_{V^{2,2}_{\alpha_j}(T)}, &\mbox{if}\ r_{j,T} > 0,
    \end{cases}    
  \]
  and an analogue to \eqref{eq:int_error_nonconvex} for the present case follows from
  $h = c_I^{-1/2} d_I^{1/2}$ and $r_{j,T} \ge ch^2 \ge c c_I^{-1}d_I$ if $r_{j,T}>0$.
  In case of $\ell=0$, we moreover have to exploit $h d_I^{-1/2} \le c c_I^{-1/2}$.

  Next, we derive interpolation error estimates for the dual problem.
  In case of $J=0,\ldots,I-3$ we obtain with standard interpolation error estimates, Lemma~\ref{lem:weighted_reg_h2} (i) and the property \eqref{eq:relation_sigma_dJ}
  \begin{align}\label{eq:int_error_w_interior_nonconvex}
   \|w-I_h w\|_{H^1(\Omega_J)} &\le c h d_J^{1/2} \|\nabla^2 w\|_{L^2(\Omega_J')} \le c h d_J^{-1/2}\left(\|\nabla w\|_{L^2(\Omega_J'')} + \|\varphi\|_{L^2(\Omega_J'')}\right).
  \end{align}
  In case of $J=I-2,I-1,I$ the function $w$ is less regular, in particular it does not belong to $H^2(\Omega_J)$ if $\Omega$ is non-convex.
  Instead, we exploit the $H^{3/2-\kappa}(\Omega)$-regularity of $w$ and obtain
  \begin{equation}\label{eq:int_error_w_boundary_nonconvex_start}
   \|w-I_h w\|_{H^1(\Omega_J)} \le c h^{1/2-\kappa} d_I^{1/4-\kappa/2} \|w\|_{H^{3/2-\kappa}(\Omega)}
 \end{equation}
 with $\kappa \in (0,1/2)$.
 To bound the norm of $w$ on the right-hand side, we apply the shift-theorem 
 from \cite[Theorem 0.5(b)]{JK95} to get
  \begin{equation}\label{eq:reg_w_h32}
   \|w\|_{H^{3/2-\kappa}(\Omega)}\le c\|\sigma^{-1}\varphi\|_{H^{-1/2-\kappa}(\Omega)}
   =c\sup_{\genfrac{}{}{0pt}{}{v\in H^{1/2+\kappa}_0(\Omega)}{v\ne0}} \frac{\int_{\Omega}\sigma^{-1}\varphi v\mathrm dx}{\|v\|_{H^{1/2+\kappa}_0(\Omega)}}.
  \end{equation}
  With the Cauchy-Schwarz inequality we obtain $\int_{\Omega}\sigma^{-1}\varphi v\mathrm dx\le c\|\varphi\|_{L^2(\Omega)}\|\sigma^{-1}v\|_{L^2(\Omega)}$.
  It remains to bound the latter factor by the $H^{1/2+\kappa}_0(\Omega)$-norm of $v$.  
  From $\|\sigma^{-1}v\|_{L^2(\Omega)} \le c d_I^{-1}\|v\|_{L^2(\Omega)}$ and \eqref{eq:poincare} we conclude  by an interpolation argument the estimate $\|\sigma^{-1} v\|_{L^2(\Omega)} \le d_I^{-1/2+\kappa} \|v\|_{H^{1/2+\kappa}_0(\Omega)}$.
  After insertion into \eqref{eq:reg_w_h32} we obtain from \eqref{eq:int_error_w_boundary_nonconvex_start} and the property $d_I=c_Ih^2$ the estimate
  \begin{equation}\label{eq:int_error_w_boundary_nonconvex}
    \|w-I_h w\|_{H^1(\Omega_J)} \le c h^{1/2-\kappa}d_I^{-1/4+\kappa/2}
    \le c c_I^{-1/4+\kappa/2}.
  \end{equation}
  
  We can now insert the estimates derived above into \eqref{eq:splitting_nonconvex}.
  First, we consider the sum over $J=0,\ldots,I-3$ only and obtain from \eqref{eq:int_error_nonconvex} and \eqref{eq:int_error_w_interior_nonconvex}
  as well as Lemma \ref{lem:weighted_reg_h1} (i) and $\|\varphi\|_{L^2(\Omega)}=1$
  \begin{align}\label{eq:product_interior}
   &\sum_{J=0}^{I-3} \|u-u_h\|_{H^1(\Omega_J)} \|w-I_h w\|_{H^1(\Omega_J)} \nonumber\\
   &\quad \le c \left(h^2 d_I^{-\overline\alpha}|u|_{V^{2,2}_{\vec\alpha}(\Omega)} + h d_I^{-1/2}\|\sigma^{-1}(u-u_h)\|_{L^2(\Omega)}\right)\left(\|\nabla w\|_{L^2(\Omega)} + \|\varphi\|_{L^2(\Omega)}\right) \nonumber\\
   &\quad \le c h^{2-2\overline\alpha} |u|_{V^{2,2}_{\vec\alpha}(\Omega)} + cc_I^{-1/2} \|\sigma^{-1}(u-u_h)\|_{L^2(\Omega)}.
  \end{align}
  For the sum over $J=I-2,I-1,I$ we use \eqref{eq:int_error_nonconvex} and \eqref{eq:int_error_w_boundary_nonconvex} instead and end up with
  \begin{align}
   \sum_{J=I-2}^{I}\|u-u_h\|_{H^1(\Omega_J)} \|w-I_h w\|_{H^1(\Omega_J)} &\le c \left(h d_I^{1/2-\overline\alpha}|u|_{V^{2,2}_{\vec\alpha}(\Omega)} + c_I^{-1/4+\kappa/2} \|\sigma^{-1}(u-u_h)\|_{L^2(\Omega)}\right) \nonumber\\
   &\le c h^{2-2\overline\alpha} |u|_{V^{2,2}_{\vec\alpha}(\Omega)} + cc_I^{-1/4+\kappa/2} \|\sigma^{-1}(u-u_h)\|_{L^2(\Omega)}.\label{eq:product_boundary}
  \end{align}  
  The inequalities \eqref{eq:product_interior} and \eqref{eq:product_boundary} together with \eqref{eq:splitting_nonconvex}
  yield
  \[
   \|\sigma^{-1}(u-u_h)\|_{L^2(\Omega)} \le c h^{2-2\overline\alpha} |u|_{V^{2,2}_{\vec\alpha}(\Omega)} + cc_I^{-1/4+\kappa/2} \|\sigma^{-1}(u-u_h)\|_{L^2(\Omega)}.
  \]
  As in the proof for the convex case, we set $c_I$ sufficiently large such that $c c_I^{-1/4+\kappa/2}
  \le 1/2$ ($\kappa\in (0,1/2)$) and we may kick back the latter term on the right-hand side to the left-hand side.
  To finish the proof, it just remains to insert the definition of the weight $\vec\alpha$ and to apply Lemma~\ref{lem:V2p_reg}. By our construction we have $\overline\alpha:=1-\overline\lambda+\varepsilon$
  which leads together with \eqref{eq:mod_nonconvex} to the desired estimate.
 \end{proof}

\section{Approximation of the surface flux}\label{sec:surfaceflux}
\subsection{Error estimates}
In the present section we apply the weighted finite element error estimates of Section \ref{sec:weighted_estimate} to derive error estimates for certain numerical approximations
of the surface flux of the solution of the Poisson equation.
\begin{theorem}\label{thm:estimate_normal_deriv}
 Let $f\in C^{0,\sigma}(\overline\Omega)$ with some $\sigma\in(0,1)$ if $\Omega$ is convex, and $f\in L^2(\Omega)$ if $\Omega$ is non-convex. 
 Denote by $u\in H^1_0(\Omega)$ and $u_h\in V_{0h}$ the solutions of \eqref{eq:weak_form_1} and \eqref{eq:fe_formulation}, respectively.
 Assume that the family of meshes $\{\mathcal T_h\}_{h>0}$ is refined according to \eqref{eq:ref_cond}.
 Then for any $\varepsilon>0$ 
 the following error estimate is fulfilled:
 \[
 \|\partial_n(u-u_h)\|_{L^2(\Gamma)} \le c h^{\min\{2,-1+2\bar\lambda-2\varepsilon\}} |\ln h|^{3/2} \times \begin{cases}
 \|f\|_{C^{0,\sigma}(\overline\Omega)}, &\mbox{if}\ \Omega\ \mbox{is convex},\\
 \|f\|_{L^2(\Omega)}, &\mbox{if}\ \Omega\ \mbox{is non-convex},
 \end{cases}
 \] 
 where $\overline\lambda :=\min_{j\in\C}\lambda_j$.
\end{theorem}
\begin{proof}
 We start with the estimate in the convex case and comment on the non-convex case later. Let $e_h:=u-u_h$. The reference elements for elements on the boundary and in the domain are denoted by $\hat E:=(0,1)$ and $\hat T:=\text{conv}\{(0,0),(1,0),(0,1)\}$, respectively.
 For an arbitrary function $v\colon T\to\mathbb R$ we use the
 notation $\hat v(\hat x)= v(F_T(\hat x))$,
 where $F_T\colon \hat T\to T$ is the affine reference transformation. For each $E\in\partial\mathcal T_h$ with associated element $T\in\mathcal T_h$ such that $\bar E=\bar T\cap\Gamma$, we obtain
 \begin{align*}
 \|\partial_{n} e_h\|_{L^2(E)}
 &\le c h_T^{-1/2} \|\partial_{\hat n} \hat e_h\|_{L^2(\hat E)}\le c h_T^{-1/2} \|\hat e_h\|_{H^{2}(\hat T)}\le c h_T^{-1/2}(|\hat e_h|_{H^{1}(\hat T)}+|\hat u|_{H^2(\hat T)})\\
 &\le c (h_T^{-1/2}|e_h|_{H^{1}(T)}+h_T^{1/2}|u|_{H^2(T)}),
 \end{align*}
 where we applied a standard trace theorem, the fact that $\hat u_h$ is piecewise linear, and the Poincar\'{e} inequality. By introducing the Lagrange interpolant $I_hu$ as an intermediate function for the first term, in combination with an inverse inequality, we deduce
 \begin{align}
 	\|\partial_{n} e_h\|_{L^2(E)}&\le c(h_T^{-1/2}|u-I_hu|_{H^{1}(T)}+h_T^{-3/2}(\|u-I_hu\|_{L^2(T)}+\|u-u_h\|_{L^2(T)})+h_T^{1/2}|u|_{H^2(T)})\notag\\
 	&\le c (h_T^{-3/2}\|u-u_h\|_{L^2(T)}+h_T^{1/2}|u|_{H^2(T)}),\label{eq:normalelement}
 \end{align}
 where we inserted a standard interpolation error estimate in the last step. Let $S_h$ denote the strip of elements at the boundary. Using $h_T\sim h^2$ if $\rho_T=0$, and the definition of the regularized distance function $\sigma$, we can show that
 \begin{align}
 	\|\partial_{n} e_h\|_{L^2(\Gamma)}
 	&\le c\left(\sum_{T\subset S_h}(\|\sigma^{-3/2}(u-u_h)\|_{L^2(T)}+h^2\|\sigma^{-1/2} \nabla^2u\|_{L^2(T)})^2\right)^{1/2}\notag\\
 	&\le c\left(\|\sigma^{-3/2}(u-u_h)\|_{L^2(\Omega)}+h^2\|\sigma^{-1/2} \nabla^2u\|_{L^2(\Omega)}\right)\label{eq:conv}.
 \end{align}
 The first assertion now follows from Theorem \ref{thm:weighted_estimate}, \eqref{eq:weightedreg} and \eqref{eq:estimatesigma}.
 
 In the non-convex case, we can not use the $H^2(T)$-regularity of $u$ if $r_{j,T}:=\dist(T,\boldsymbol c_j)=0$, $j\in\mathcal{C}_{non}$. Instead, we introduce $\alpha_j:=\max\{0,1-\lambda_j+\varepsilon\}$, $j\in\mathcal{C}$. Then, for each $E\in\partial\mathcal T_h$, whose corresponding element $T\in\mathcal T_h$ with $\bar E=\bar T\cap\Gamma$ fulfills $r_{j,T}=0$, $j\in\mathcal{C}_{non}$, we obtain
 \begin{align*}
   \|\partial_{n} e_h\|_{L^2(E)}
   &\le c h_T^{-1/2} \|\partial_{\hat n} \hat e_h\|_{L^2(\hat E)}
   \le c h_T^{-1/2}\|\hat e_h\|_{W^{2,q}(\hat T)}\le c h_T^{-1/2}(\|\hat e_h\|_{W^{1,q}(\hat T)}+|\hat u|_{W^{2,q}(\hat T)}) ,
 \end{align*}
 where we applied a standard trace theorem, which holds
 for any $q > 4/3$.
 For the first term we use the embedding $H^1(\hat T)\hookrightarrow W^{1,q}(\hat T)$ and the 
 Poincar\'{e} inequality. The second term is treated with 
 the embedding $V^{0,2}_{\alpha_j}(\hat T) \hookrightarrow  L^q(\hat T)$,
 which holds for any $\alpha_j < 1/2$ if $q$ is sufficiently close to $4/3$.
 From this we infer
 \begin{equation}\label{eq:V22_trafo}
   \|\partial_{n} e_h\|_{L^2(E)}\le c h_T^{-1/2}(|e_h|_{H^{1}(\hat T)}+|\hat u|_{V^{2,2}_{\alpha_j}(\hat T)})\le c (h_T^{-1/2}|e_h|_{H^{1}(T)}+h_T^{1/2-\alpha_j}|u|_{V^{2,2}_{\alpha_j}(T)}).
 \end{equation}
 Note, that the weight $\hat r$ contained in the space $V^{2,2}_{\alpha_j}(\hat T)$ is 
   related to the corner $(0,0)$ of $\hat T$. Without loss of generality we may define $F_T$ in such a way that
   $F_T(0,0)=\boldsymbol c_j$. This implies the property $\hat r := |\hat x| \sim h_T^{-1}\,r_j$ that we
   used in the last step of the estimate above.
 Now, as in \eqref{eq:normalelement}, we conclude that
 \[
	 \|\partial_{n} e_h\|_{L^2(E)}\le c (h_T^{-1/2}|u-I_hu|_{H^{1}(T)}+h_T^{-3/2}(\|u-I_hu\|_{L^2(T)}+\|u-u_h\|_{L^2(T)})+h_T^{1/2-\alpha_j}|u|_{V^{2,2}_{\alpha_j}(T)}).
 \]
 The resulting terms for the interpolation error can be treated with the estimate 
 \begin{equation}\label{eq:int_error_Sh}
 	h_T\|\nabla (u-I_h u)\|_{L^2(T)} +  \|u-I_h u\|_{L^2(T)} \le c h_T^{2-\alpha_j}\|u\|_{V^{2,2}_{\alpha_j}(T)}
 \end{equation}
 proved in \cite[Section 3.3]{ASW96}. This yields
 \[
 	\|\partial_{n} e_h\|_{L^2(E)}\le c(h_T^{-3/2}\|u-u_h\|_{L^2(T)}+h_T^{1/2-\alpha_j}\|u\|_{V^{2,2}_{\alpha_j}(T)}).
 \]
 For each $E \in \partial\mathcal{T}_h$ with positive distance to the non-convex corners, we obtain analogously to \eqref{eq:normalelement}
 \[
 	\|\partial_{n} e_h\|_{L^2(E)}\le c(h_T^{-3/2}\|u-u_h\|_{L^2(T)}+h_T^{1/2}|u|_{H^2(T)}).
 \]
 After having noted that $h_T^{\alpha_j}\leq r_{j,T}^{\alpha_j}$ if $r_{j,T}>0$, and that the semi-norms of $V^{2,2}_{\alpha_j}(\Omega_R^j)$ and $H^2(\Omega_R^j)$ coincide if $\alpha_j=0$, we can sum up the previous two inequalities and arrive similar to \eqref{eq:conv} at
 \[
   \|\partial_{n} e_h\|_{L^2(\Gamma)}\le c(\|\sigma^{-3/2}(u-u_h)\|_{L^2(\Omega)}+h^{\min\{1,-1+2\bar{\lambda}-2\epsilon\}}\|u\|_{V^{2,2}_{\vec{\alpha}}(\Omega)}).
 \]
 The assertion in case of non-convex domains is finally a consequence of
 Theorem~\ref{thm:weighted_estimate_nonconvex} and Lemma~\ref{lem:V2p_reg}.
\end{proof}

A second approach to approximate the surface flux is given by the concept of a discrete variational normal derivative.
This has several applications in optimal boundary control, see Section \ref{sec:control},
or for the approximation of Steklov-Poincar\'e operators used for instance 
in domain decomposition techniques \cite{AL85,QV99,XZ97}.
For a given function $u_h\in V_{0h}$ solving \eqref{eq:fe_formulation}, we define its discrete variational normal derivative as the object $\partial_n^h u_h\in V_h^\partial:=\operatorname{Tr}(V_h)$ (the trace space of $V_h$) fulfilling
\begin{equation}\label{eq:def_var_normal}
 (\partial_n^h u_h, w_h)_{L^2(\Gamma)} = (\nabla u_h,\nabla w_h)_{L^2(\Omega)} - (f, w_h)_{L^2(\Omega)}\qquad \forall w_h\in V_h.
\end{equation}
Note that the normal derivative $\partial_n u$ of $u\in H^1_0(\Omega)$ solving \eqref{eq:weak_form_1} fulfills due to Green's identity
\[
	(\partial_n u, w)_{L^2(\Gamma)} = (\nabla u,\nabla w)_{L^2(\Omega)} - (f, w)_{L^2(\Omega)}\qquad \forall w\in H^1(\Omega),
\]
such that
\begin{equation}\label{eq:normalerror}
	(\partial_n u-\partial_n^h u_h, w_h)_{L^2(\Gamma)} = (\nabla (u-u_h),\nabla w_h)_{L^2(\Omega)}\qquad \forall w_h\in V_h.
\end{equation}
Using the previous expression and the weighted estimates from Section \ref{sec:weighted_estimate}, we show error estimates for the discrete variational normal derivative in the following.
\begin{theorem}\label{thm:estimate_var_normal_deriv}
 Let $u\in H^1_0(\Omega)$ denote the solution of \eqref{eq:weak_form_1}. Assume further that $f\in C^{0,\sigma}(\overline\Omega)$, $\sigma\in (0,1)$, if $\Omega$ is convex, and that
 $f\in L^2(\Omega)$ if $\Omega$ is non-convex. Provided that the family of meshes $\{\mathcal T_h\}_{h>0}$ satisfy \eqref{eq:ref_cond}, 
 the solution $\partial_n^h u_h$ of \eqref{eq:def_var_normal} fulfills for any $\varepsilon>0$ the error estimate
 \[
  \|\partial_n u - \partial_n^h u_h\|_{L^2(\Gamma)} \le c h^{\min\{2,-1+2\overline\lambda-\varepsilon\}}|\ln h|^{3/2}\times\begin{cases}
  \|f\|_{C^{0,\sigma}(\overline\Omega)}, &\mbox{if}\ \Omega\ \mbox{is convex},\\
  \|f\|_{L^2(\Omega)}, &\mbox{if}\ \Omega\ \mbox{is non-convex},
  \end{cases}
 \]
 where $\overline\lambda :=\min_{j\in\C}\lambda_j$.
\end{theorem}
\begin{proof}
We start with introducing a Cl\'ement type interpolation operator: Let $\varphi_i$ with $i=1,\ldots,N_h:=\dim(V_h^\partial)$ denote the nodal basis functions of $V_h^\partial$. The interpolation operator $C_h:L^1(\Gamma)\rightarrow V_h^\partial$ is given by
\[
	C_hv=\sum_{j=1}^{N_h} v_i \varphi_i \quad\text{with }v_i:= \frac{1}{\left|\supp \varphi_i\right|}\int_{\supp \varphi_i} v\, \mathrm ds.
\]
Let $S_E$ denote the set of elements in $\partial \mathcal{T}_h$ sharing at least one vertex with $E$. A short calculation shows that
\[
	\|C_hv\|_{L^2(E)}\le c \|v\|_{L^2(S_E)}\quad\text{and}\quad C_h p_0 = p_0 \text{ in } E\quad \forall p_0\in\mathcal{P}_0(S_E).
\]

We now turn our attention to the proof of the assertion. Introducing $C_h\partial_n u$ as an intermediate function, we immediately obtain
\[
\|\partial_n u - \partial_n^h u_h\|^2_{L^2(\Gamma)}= (\partial_n u - \partial_n^hu_h,\partial_n u -C_h\partial_n u)_{L^2(\Gamma)}+(\partial_n u - \partial_n^hu_h,C_h(\partial_n u - \partial_n^hu_h))_{L^2(\Gamma)}.
\]
Using the Cauchy-Schwartz inequality and the stability of $C_h$ in $L^2(\Gamma)$, we can continue with
\begin{equation}\label{eq:startnormal}
\|\partial_n u - \partial_n^h u_h\|_{L^2(\Gamma)}\leq \|\partial_n u -C_h\partial_n u\|_{L^2(\Gamma)}+\sup_{\genfrac{}{}{0pt}{}{\varphi_h\in V_h^\partial}{\|\varphi_h\|_{L^2(\Gamma)=1}}}\left|(\partial_n u - \partial_n^hu_h,\varphi_h)_{L^2(\Gamma)}\right|.
\end{equation}
Subsequently, we distinguish similar to the proof of Theorem \ref{thm:estimate_normal_deriv} between convex and non-convex domains. We first consider the convex case, and start with showing an estimate for the interpolation error. For $E\in\partial\mathcal T_h$, let $T$ be the element in $\mathcal T_h$ with $\bar T\cap \Gamma=\bar E$. Moreover, we define $D_E$ as the set of all elements in $\mathcal{T}_h$ sharing at least one vertex with $E$. The reference configurations $S_{\hat E}$ and $D_{\hat E}$ are given by $S_{\hat E}=F_T^{-1}(S_E)$ and $D_{\hat E}=F_T^{-1}(D_E)$, see the proof of Theorem \ref{thm:estimate_normal_deriv} for the definition of the mapping $F_T$. If the element $E$ has a positive distance to each corner, we introduce $\partial_{n}p\in \mathcal{P}_0(S_E)$ as an intermediate function. Here, $p$ denotes an arbitrary polynomial in $\mathcal P_1(D_E)$.
Afterwards, we employ the aforementioned stability of $C_h$, a standard trace theorem on the reference configuration, and the Bramble-Hilbert Lemma. This yields
\begin{align}
\|\partial_n u - C_h\partial_n u\|_{L^2(E)}
&\le c\|\partial_{n} u - \partial_{n} p\|_{L^2(S_{E})}
 \le c h_T^{-1/2}\|\partial_{\hat n} \hat u - \partial_{\hat n} \hat p\|_{L^2(S_{\hat E})}\notag\\
&\le c h_T^{-1/2}\|\hat u - \hat p\|_{H^{2}(D_{\hat E})} \le c h_T^{-1/2}|\hat u|_{H^{2}(D_{\hat E})}.\label{eq:intelement}
\end{align}
If the element $E$ has contact to a corner, we similarly obtain without introducing an intermediate function
\begin{align}
\|\partial_n u - C_h\partial_n u\|_{L^2(E)}
&\le c\|\partial_{n} u\|_{L^2(S_{E})}
\le c h_T^{-1/2}\|\partial_{\hat n} \hat u\|_{L^2(S_{\hat E})}\notag\\
&\le c h_T^{-1/2}\|\hat u\|_{H^{2}(D_{\hat E})} \le c h_T^{-1/2}|\hat u|_{H^{2}(D_{\hat E})}.\label{eq:intelement2}
\end{align}
In the last step, we used that the zero function is a linear interpolant of $\hat u$ on $D_{\hat E}$ because of the homogeneous boundary conditions of $\hat u$ on $S_{\hat E}$ (which contains a kink due to the convex corner). Collecting the results from \eqref{eq:intelement} and \eqref{eq:intelement2}, after having transformed everything to the world configuration, yields
\begin{equation}\label{eq:normalint}
\|\partial_n u - C_h\partial_n u\|_{L^2(\Gamma)}
\le c\left(\sum_{\stackrel{T\in\mathcal{T}_h}{\rho_T=0}}h_T\| \nabla^2u\|_{L^2(D_{E})}^2\right)^{1/2}
\le c h^2\|\sigma^{-1/2} \nabla^2u\|_{L^2(\Omega)},
\end{equation}
where we used $h_T\sim h^2$, which holds for elements $T$ in the direct vicinity of the boundary, and the definition of the regularized distance function $\sigma$. Next, we show an estimate for the second term in \eqref{eq:startnormal}. To this end, let $S_h$ denote the strip of elements at the boundary. Furthermore, we introduce the zero-extension $\tilde B_h\colon V_h^\partial\to V_h$ defined in such a way
that $\tilde B_h v_h$ vanishes in the interior nodes for arbitrary $v_h\in
V_h^\partial$, and hence it is always supported in the boundary strip $S_h$. This extension operator admits the stability estimate 
\begin{equation}\label{eq:stability_tilde_Bh}
\|\nabla \tilde B_hv_h\|_{L^2(S_h)} \le c h^{-1}\|v_h\|_{L^2(\Gamma)} \qquad\forall v_h\in V_h^\partial,
\end{equation}
which is proved in \cite[Lemma 3.3]{MRV13}. Using \eqref{eq:normalerror}, \eqref{eq:stability_tilde_Bh}, and $\|\varphi_h\|_{L^2(\Gamma)}=1$, we conclude
\begin{align*}
\left|(\partial_n u - \partial_n^hu_h,\varphi_h)_{L^2(\Gamma)}\right|
&=\left|(\nabla(u-u_h),\nabla \tilde B_h\varphi_h)_{L^2(S_h)}\right|
\le\|\nabla(u-u_h)\|_{L^2(S_h)} \|\nabla \tilde B_h \varphi_h\|_{L^2(S_h)}\notag\\
&\le c h^{-1}\|\nabla(u-u_h)\|_{L^2(S_h)}.
\end{align*}
Introducing the Lagrange interpolant $I_hu$ as an intermediate function, in combination with an inverse inequality, yields
\begin{align}
\big|(\partial_n u - \partial_n^hu_h&,\varphi_h)_{L^2(\Gamma)}\big|\notag\\
&\le c \left(h^{-1} \|\nabla(u-I_h u)\|_{L^2(S_h)} + h^{-3} \|u-I_h u\|_{L^2(S_h)} + \|\sigma^{-3/2}(u-u_h)\|_{L^2(\Omega)}\right)\notag\\
&\le c \left(h^2\|\sigma^{-1/2} \nabla^2u\|_{L^2(\Omega)}+\|\sigma^{-3/2}(u-u_h)\|_{L^2(\Omega)}\right),\label{eq:secondnormal}
\end{align}
where we used the definition of $\sigma$ and a standard interpolation error estimate. The assertion in the convex case now follows from \eqref{eq:startnormal}, \eqref{eq:normalint}, \eqref{eq:secondnormal}, Theorem \ref{thm:weighted_estimate}, \eqref{eq:weightedreg} and \eqref{eq:estimatesigma}.

In the non-convex case, we have to slightly modify the proof due to the lack of regularity. We again consider the interpolation error in \eqref{eq:startnormal} at first. For every $E\in\partial \mathcal T_h$ such that $D_E$ touches a corner $\boldsymbol c_j$, $j\in\mathcal C_{non}$, we obtain
analogously to \eqref{eq:intelement2} for $q> 4/3$
\[
\|\partial_n u - C_h\partial_n u\|_{L^2(E)}
\le c\|\partial_{n} u\|_{L^2(S_{E})}
\le c h_T^{-1/2}\|\partial_{\hat n} \hat u\|_{L^2(S_{\hat E})}\le c h_T^{-1/2}\|\hat u\|_{W^{2,q}(D_{\hat E})}.
\]
Next, we apply the embedding $V^{2,2}_{\alpha_j}(D_{\hat E}) \hookrightarrow  W^{2,q}(D_{\hat E})$,
which holds for any $\alpha_j < 1/2$ if $q$ is sufficiently close to $4/3$.
In the present situation we choose $\alpha_j:=\max\{0,1-\lambda_j+\varepsilon\}$, $j\in\mathcal{C}$,
with sufficiently small $\varepsilon >0$.
Together with the transformation formula used already in \eqref{eq:V22_trafo} we deduce
\[
\|\partial_n u - C_h\partial_n u\|_{L^2(E)}\le ch_T^{-1/2}\|\hat u\|_{W^{2,q}(D_{\hat E})} 
\le ch_T^{-1/2}\|\hat u\|_{V^{2,2}_{\alpha_j}(D_{\hat E})} 
\le c h_T^{1/2-\alpha_j} \|u\|_{V^{2,2}_{\alpha_j}(D_E)}.
\]
If $D_E$ has a positive distance to the corner, it is possible to use \eqref{eq:intelement} again. After having noted that $h_T^{\alpha_j}\leq r_{j,T}^{\alpha_j}$ if $r_{j,T}>0$, and that the semi-norms of $V^{2,2}_{\alpha_j}(\Omega_R^j)$ and $H^2(\Omega_R^j)$ coincide if $\alpha_j=0$, we can sum up the previous results analogously to \eqref{eq:normalint} to conclude
\begin{equation}\label{eq:intnonc}
\|\partial_n u - C_h\partial_n u\|_{L^2(\Gamma)}\le c h^{\min\{1,-1+2\bar{\lambda}-2\epsilon\}}\|u\|_{V^{2,2}_{\vec{\alpha}}(\Omega)}.
\end{equation}
As for \eqref{eq:secondnormal}, but now using the interpolation error estimate \eqref{eq:int_error_Sh} if $T$ touches a non-convex corner, we deduce
\begin{equation}\label{eq:normalnonc}
\left|(\partial_n u - \partial_n^hu_h,\varphi_h)_{L^2(\Gamma)}\right|\le c \left(h^{\min\{1,-1+2\bar{\lambda}-2\epsilon\}}\|u\|_{V^{2,2}_{\vec{\alpha}}(\Omega)}+\|\sigma^{-3/2}(u-u_h)\|_{L^2(\Omega)}\right).
\end{equation}
The assertion for non-convex domains is now a consequence of \eqref{eq:startnormal}, \eqref{eq:intnonc}, \eqref{eq:normalnonc}, Theorem \ref{thm:weighted_estimate_nonconvex} and Lemma \ref{lem:V2p_reg}.
\end{proof}

\subsection{Numerical experiments}\label{sec:experiments}

In this section we will show that the error estimate derived in Theorem \ref{thm:estimate_normal_deriv} is sharp.
Therefore, we construct the following benchmark problem.
We introduce a family of computational domains 
\[
 \Omega_\omega:= (-1,1)^2 \cap\{(r\cos\varphi, r\sin\varphi)\colon r\in (0,\infty),\ \varphi\in (0,\omega)\},\qquad\omega\in [\pi/2,2\pi)
\]
with $r$ and $\varphi$ denoting the polar coordinates located at the origin.
For these domains the corner with the largest opening angle $\omega$ is the origin, and
the corresponding singular exponent is $\bar\lambda:= \pi/\omega$.

The problem we consider in the numerical experiment is the problem whose exact solution is 
\[
 u(x_1,x_2) := r^{\bar\lambda}(x_1,x_2) \sin(\bar\lambda\varphi(x_1,x_2))(1-x_1^2)(1-x_2^2).
\]
With a simple computation we obtain the corresponding right-hand side $f$. The homogeneous Dirichlet boundary conditions are fulfilled by construction.

The meshes used for the computation are constructed in the following way. 
We start with a coarse initial mesh consisting of $2$ ($\omega=\pi/2$),
$3$ ($\omega=2\pi/3,\,3\pi/4$),
$5$ ($\omega=5\pi/4$), $6$ ($\omega=3\pi/2$) or $7$ ($\omega=7\pi/4$) elements.
The fine meshes are obtained by $N$ global steps of a newest-vertex bisection strategy \cite{Bae91}.
Afterwards, this strategy is successively applied to all elements violating the refinement condition \eqref{eq:ref_cond}.  
Two meshes with $h=2^{-3}$ for the domains $\Omega_{\pi/2}$ and $\Omega_{3\pi/4}$ are illustrated in Figure~\ref{fig:meshes}.
\begin{figure}
\begin{center}
 \subfloat[Mesh for $\Omega_{\pi/2}$]{\includegraphics[height=3.5cm]{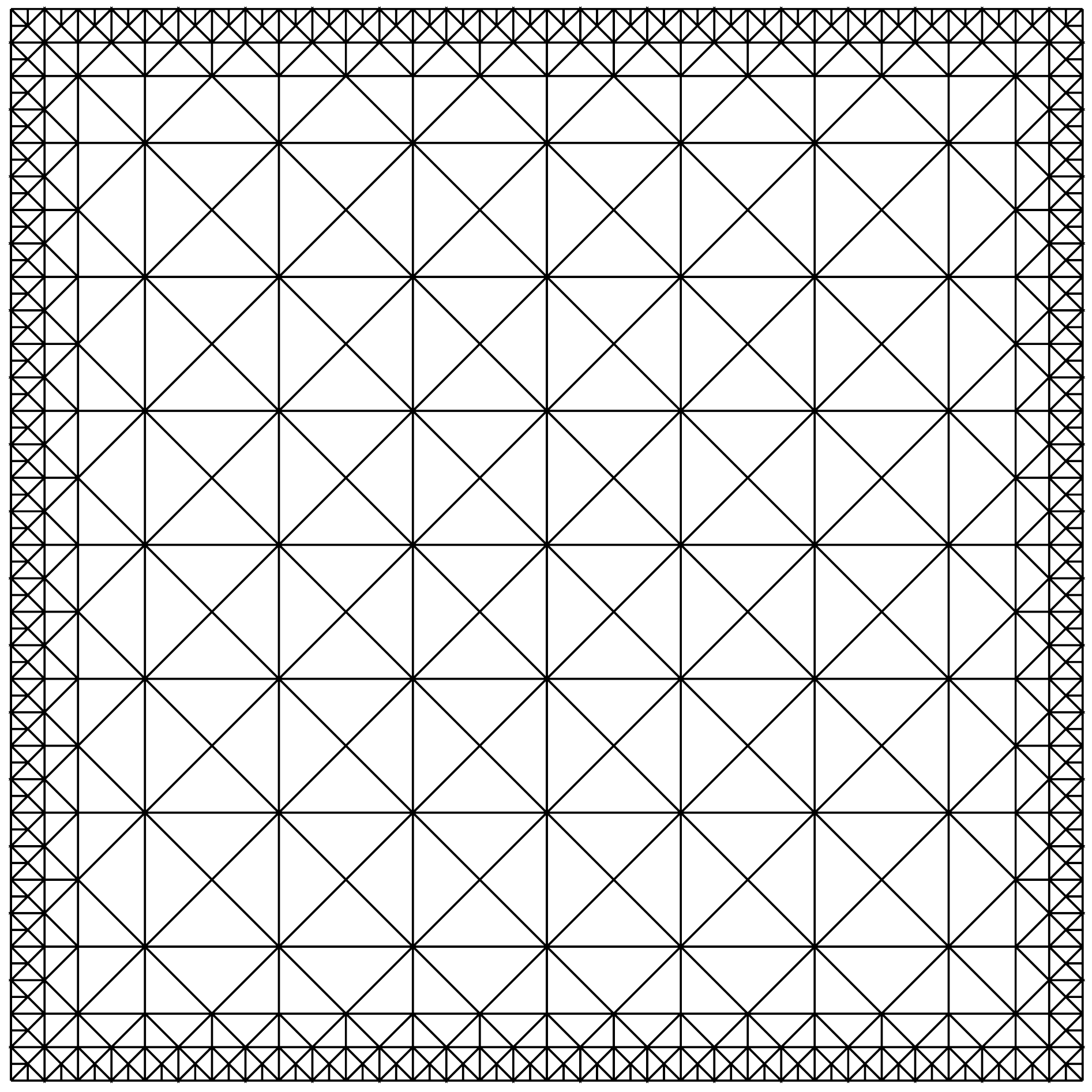}}\quad
 \subfloat[Mesh for $\Omega_{3\pi/4}$]{\includegraphics[height=3.5cm]{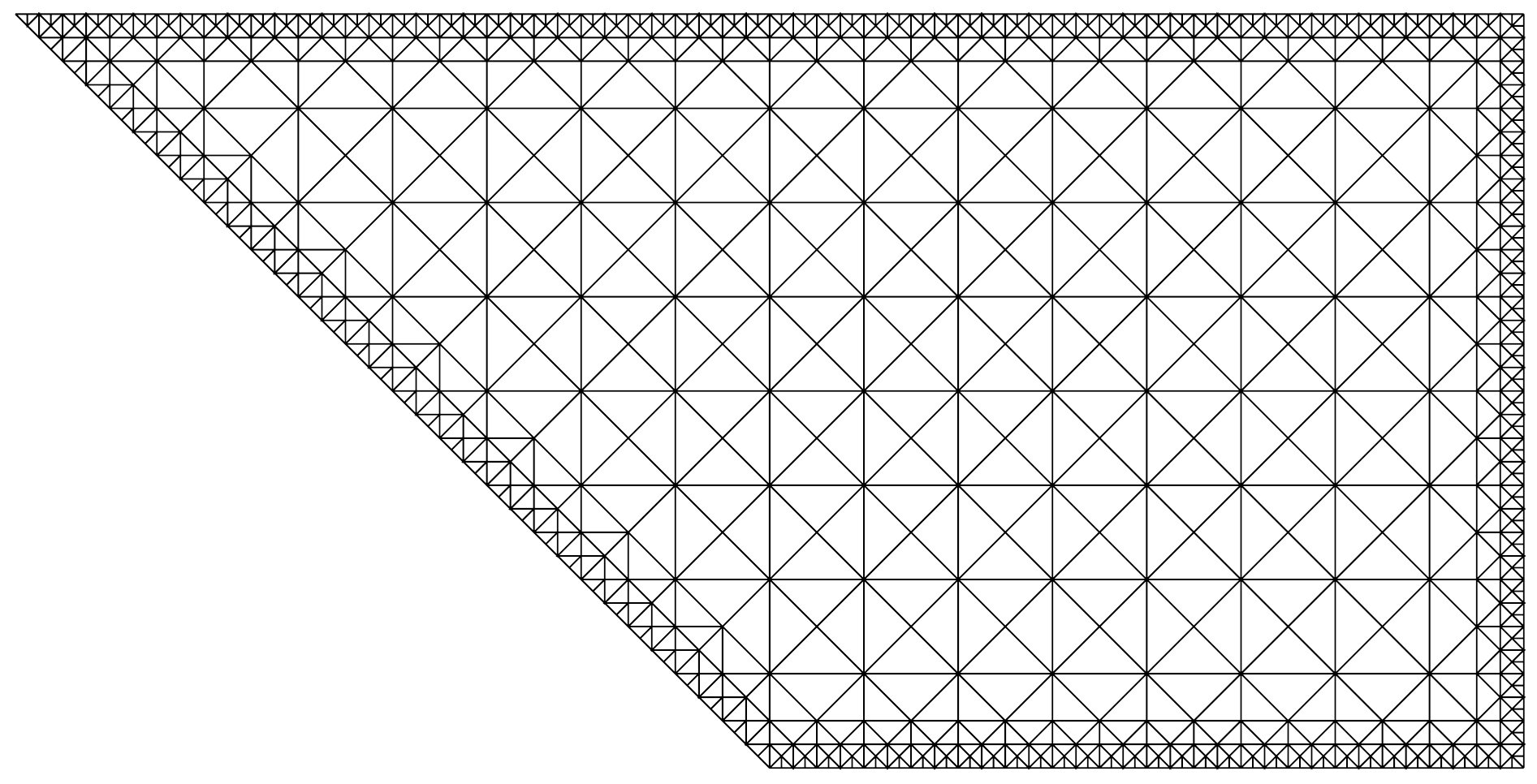}}
 \end{center}
 \caption{Meshes satisfying the refinement criterion \eqref{eq:ref_cond}.}
 \label{fig:meshes}
\end{figure}

After computing the finite element approximation $u_h$ from \eqref{eq:fe_formulation} the error term
$\|\partial_n (u-u_h)\|_{L^2(\Gamma)}$ is evaluated.
The results of our computations are summarized in Tables \ref{tab:experiment3} and \ref{tab:experiment4}.
In all cases we observe that the convergence rates of
Theorem \ref{thm:estimate_normal_deriv} are confirmed. 
The convergence rate $2$ is observed for the angles $90^\circ$ and $120^\circ$,
but not for $135^\circ$. 
This confirms the fact that $120^\circ$ is the limiting case.

\begin{table}[htb]
\begin{center}
\begin{tabular}{rrrr}
\toprule
\multicolumn{1}{l}{} & \multicolumn{3}{c}{$\|\partial_n(u-u_h)\|_{L^2(\Gamma)}$\ (EOC)} \\ \cmidrule{2-4}
\multicolumn{1}{r}{$h$} & \multicolumn{1}{c}{$\omega=90^\circ$}  & \multicolumn{1}{c}{$\omega=120^\circ$} 
& \multicolumn{1}{c}{$\omega=135^\circ$} 
\\ \midrule
 $2^{-4}$ & 5.09e-1 (1.86)  & 6.71e-1 (1.25) & 5.32e-1 (1.63)  \\ 
 $2^{-5}$ & 1.30e-1 (1.97)  & 2.00e-1 (1.75) & 1.41e-1 (1.91)  \\ 
 $2^{-6}$ & 3.27e-2 (1.99)  & 5.37e-2 (1.89) & 3.76e-2 (1.91)  \\ 
 $2^{-7}$ & 8.20e-3 (2.00)  & 1.38e-2 (1.96) & 1.02e-2 (1.89)  \\ 
 $2^{-8}$ & 2.05e-3 (2.00)  & 3.50e-3 (1.98) & 2.82e-3 (1.85)  \\ 
 $2^{-9}$ & 5.12e-4 (2.00)  & 8.87e-4 (1.98) & 8.01e-4 (1.81)  \\ 
 $2^{-10}$ & 1.28e-4 (2.00)  & 2.24e-4 (1.98) & 2.33e-4 (1.78) \\
 $2^{-11}$ & 3.20e-5 (2.00)  & 5.68e-5 (1.98) & 6.95e-5 (1.75) \\
 \midrule
 Expected: & (2.00) & (2.00) & (1.67) \\
\bottomrule
\end{tabular}
\end{center}
\caption{Experimental convergence rates for $\|\partial_n(u-u_h)\|_{L^2(\Gamma)}$ for convex domains.}
\label{tab:experiment3}
\end{table}

\begin{table}[htb]
\begin{center}
\begin{tabular}{rrrr}
\toprule
\multicolumn{1}{l}{} & \multicolumn{3}{c}{$\|\partial_n(u-u_h)\|_{L^2(\Gamma)}$\ (EOC)} \\ \cmidrule{2-4}
\multicolumn{1}{r}{$h$}  & \multicolumn{1}{c}{$\omega=225^\circ$}  & \multicolumn{1}{c}{$\omega=270^\circ$} 
& \multicolumn{1}{c}{$\omega=315^\circ$} 
\\ \midrule
 $2^{-4}$  & 6.36e-1 (1.47) & 7.11e-1 (1.36) & 9.75e-1 (0.98) \\ 
 $2^{-5}$  & 2.88e-1 (1.14) & 4.08e-1 (0.80) & 7.50e-1 (0.38) \\ 
 $2^{-6}$   & 1.73e-1 (0.73) & 3.08e-1 (0.41) & 6.73e-1 (0.16) \\ 
 $2^{-7}$  & 1.12e-1 (0.62) & 2.43e-1 (0.34) & 6.15e-1 (0.13) \\ 
 $2^{-8}$  & 7.40e-2 (0.60) & 1.93e-1 (0.33) & 5.63e-1 (0.13) \\ 
 $2^{-9}$  & 4.88e-2 (0.60) & 1.53e-1 (0.33) & 5.17e-1 (0.12) \\ 
 $2^{-10}$ & 3.22e-2 (0.60) & 1.22e-1 (0.33) & 4.76e-1 (0.12) \\
 $2^{-11}$ & 2.13e-2 (0.60) & 9.65e-2 (0.33) & 4.40e-1 (0.12) \\
 \midrule
 Expected: & (0.60) & (0.33) & (0.14)\\
\bottomrule
\end{tabular}
\end{center}
\caption{Experimental convergence rates for $\|\partial_n(u-u_h)\|_{L^2(\Gamma)}$ for non-convex domains.}
\label{tab:experiment4}
\end{table}

\section{Discretization of Dirichlet boundary control problems}\label{sec:control}

\subsection{Error estimates}
An application of the results of Theorem \ref{thm:estimate_var_normal_deriv}
are error estimates for the finite element approximation of Dirichlet boundary
control problems. 
As a model problem we investigate
\begin{equation}\label{eq:control}
 \begin{aligned}
  J(y,u) := \frac12\|y&-y_d\|_{L^2(\Omega)}^2 + \frac\alpha2\|u\|^2_{L^2(\Gamma)}\to\min!\\[.3em]
  \mbox{s.\,t.}\quad -\Delta y &= 0\quad\mbox{in}\ \Omega,\\
  y&=u\quad\mbox{on}\ \Gamma,
 \end{aligned}
\end{equation}
where $\alpha>0$ and a desired state $y_d\in L^2(\Omega)$ are given.
It is known \cite{MRV13} that the weak formulation of the system
\begin{equation}\label{eq:optimality}
\left\lbrace
\begin{aligned}
 -\Delta y&=0 \quad& -\Delta p &= y-y_d &\quad&\mbox{in}\ \Omega,\\
 y&=u & p&= 0 && \mbox{on}\ \Gamma,\\ 
 && \alpha u - \partial_n p &=0 &&\mbox{on}\ \Gamma,
 \end{aligned}
 \right.
\end{equation}
forms a necessary and sufficient optimality condition. For non-convex domains
the state equation has to be understood in the very weak sense as $y\notin
H^1(\Omega)$ in general.
In order to derive error estimates for the numerical approximation of $(y,u,p)$ we collect some regularity results in the next lemma.
\begin{lemma}\label{lem:regularity}
Let $s < \min\{2,\bar\lambda\}$ with $\bar \lambda:=\min_{j\in \mathcal{C}}\lambda_j$, and let $\vec \gamma \in \mathbb{R}^d$ statisfying $\gamma_j>\max\{0,2-\lambda_j\}$ for $j\in\mathcal C$. Then for $y_d\in H^1(\Omega)$, there holds
\[
y\in H^{s}(\Omega)\cap V^{2,2}_{\vec\gamma}(\Omega)\quad \text{and}\quad u\in H^{s-1/2}(\Gamma).
\]
\end{lemma}
\begin{proof}
	The regularity results for $y$ and $u$ in $H^s(\Omega)$ and $H^{s-1/2}(\Gamma)$, respectively, can be found in \cite[Lemma 2.2]{AMPR16}. Basically, the regularity result in weighted Sobolev spaces for the state is contained in that reference as well. However, it is not directly accessible. For that reason, we give a short proof by employing a bootstrapping argument and classical regularity results in weighted Sobolev spaces:
	After having noticed that $y-y_d$ belongs to $L^2(\Omega)$, we can
    deduce $p\in V^{2,2}_{\vec \beta}(\Omega)$ for $\vec \beta\in
    \mathbb{R}^{d}$ satisfying $\beta_j>1-\lambda_j$ and $\beta_j\ge0$ for
    $j\in\mathcal C$, see Lemma \ref{lem:V2p_reg}. Due to the
    optimality condition, $\alpha u = \partial_np$ almost everywhere on $\Gamma$, and
    trace and extension theorems in weighted Sobolev spaces from
    \cite[Theorem 1.31 and Theorem 1.32]{Nicaise1993}, we are able to show
    by classical means that $y\in V^{1,2}_{\vec\beta}(\Omega)$ (eventually by
    using a density argument). For related results and techniques, we also
    refer to \cite[Section C and Section E]{ApNiPfe2016}. Since $\beta_j$
    can always be chosen such that $\beta_j\leq \gamma_j$, we obtain from
    \cite[Theorem 3.1]{NP94}, by setting $l=1$ in this theorem, that
    $p$ belongs to $V^{3,2}_{\vec \gamma}(\Omega)$. We notice that the
    embedding $y_d\in H^{1}(\Omega)\hookrightarrow
    V^{1,2}_{\vec\varepsilon}(\Omega)$ for $\vec\varepsilon$ with
    $\varepsilon_j>0$, 
    $j\in\mathcal C$, is essential for the applicability of the theorem, see \cite[Theorem 7.1.1]{KMR97}. Similar to before, due
    to the trace and extension theorems from \cite{Nicaise1993}, we can
    finally show the assertion using \cite[Theorem 3.1]{NP94}, now by
    setting $l=0$. For the application of the theorem, it is important to
    notice that there holds $2-\lambda_j<1+\lambda_j$ as
    $\omega_j<2\pi$ for $j\in\mathcal C$ such that the range for the
    weights is non-empty.
\end{proof}

Analogous to  \cite{AMPR16, CR06, MRV13} we compute an approximation of $(y,u,p)$ obtained 
by the finite element method. Therefore, we consider a family of finite element meshes 
$\{\mathcal T_h\}_{h>0}$ refined according to \eqref{eq:ref_cond},
and seek the approximate solutions in the spaces
\begin{align*}
 V_h:=\{v_h\in C(\overline\Omega)\colon v_h|_T\in\mathcal P_1\text{ for all } T\in\mathcal T_h\}, \quad V_h^\partial:=\operatorname{Tr}(V_h),\qquad V_{0h}:=V_h\cap H^1_0(\Omega).
\end{align*}
The discrete optimality condition reads:
Find $(y_h,u_h,p_h)\in V_h\times V_h^\partial\times V_{0h}$ such that
\begin{equation}\label{eq:discrete_opt_cond}
\left\lbrace
\begin{aligned}
 y_h|_\Gamma= u_h,\quad(\nabla y_h,\nabla v_h)_{L^2(\Omega)} &= 0 &&\forall v_h\in V_{0h},\\
 (\nabla v_h,\nabla p_h)_{L^2(\Omega)} &= (y_h-y_d,v_h)_{L^2(\Omega)}&&\forall v_h\in V_{0h},\\
 (\alpha u_h - \partial_n^h p_h,w_h)_{L^2(\Gamma)} &= 0 &&\forall w_h\in V_h^\partial.
 \end{aligned}\right.
\end{equation}
Here, $\partial_n^h p_h \in V_h^\partial$ denotes
the discrete variational normal derivative of the adjoint state defined by
\begin{equation}\label{eq:discrete_normal_adjoint}
 (\partial_n^h p_h,v_h)_{L^2(\Gamma)} = (\nabla v_h,\nabla p_h)_{L^2(\Omega)} - (y_h-y_d,v_h)_{L^2(\Omega)}\quad\forall v_h\in V_h.
\end{equation}
The following basic error estimate has been shown in \cite{AMPR16}:
\begin{align}\label{eq:basic_estimate}
 &\phantom{\le c} \|u-u_h\|_{L^2(\Gamma)} + \|y-y_h\|_{L^2(\Omega)}\nonumber\\
 &\le c\left(\|u-Q_h u\|_{L^2(\Gamma)} + \|y-B_h Q_h u\|_{L^2(\Omega)} + \sup_{\psi_h\in V_h^\partial}\frac{|(\nabla p,\nabla B_h\psi_h)_{L^2(\Omega)}|}{\|\psi_h\|_{L^2(\Gamma)}}  \right).
\end{align}
The operator $B_h\colon V_h^\partial\to V_h$ denotes the discrete harmonic extension, and $Q_h\colon L^2(\Gamma)\to V_h^\partial$ the $L^2(\Gamma)$-projection. Let $R_hp\in V_{0h}$ denote the Ritz projection of $p$ defined by
\[
	(\nabla R_h p,\nabla v_h)_{L^2(\Omega)} = (\nabla p,\nabla v_h)_{L^2(\Omega)}\quad\forall v_h\in V_{0h}.
\]
Then, we obtain due to the definition of the discrete harmonic extension $B_h$ and \eqref{eq:normalerror}
\[
	(\nabla p,\nabla B_h\psi_h)_{L^2(\Omega)}=(\nabla (p-R_hp),\nabla B_h\psi_h)_{L^2(\Omega)}=(\partial_np-\partial_n^hR_hp,\psi_h)_{L^2(\Gamma)}
\]
such that the third term in \eqref{eq:basic_estimate} can be estimated by
\begin{equation}\label{eq:adjointnormal}
	\sup_{\psi_h\in V_h^\partial}\frac{|(\nabla p,\nabla B_h\psi_h)_{L^2(\Omega)}|}{\|\psi_h\|_{L^2(\Gamma)}}\le c \|\partial_np-\partial_n^hR_hp\|_{L^2(\Gamma)}.
\end{equation}
In the following, we discuss each of the different error contributions in \eqref{eq:basic_estimate} and \eqref{eq:adjointnormal}.
\begin{lemma}\label{lem:first_term}
 Assume that $y_d\in H^{1}(\Omega)$. Let $u$ be the optimal control solving \eqref{eq:optimality}.
 Then, the estimate
 \[
  \|u-Q_h u\|_{L^2(\Gamma)} \le c h^{\min\{3,-1+2\bar\lambda\}-2\varepsilon}
 \]
 holds with $\bar \lambda:=\min_{j\in \mathcal{C}}\lambda_j$ and any $\varepsilon>0$, provided that the refinement condition \eqref{eq:ref_cond} is fulfilled.
\end{lemma}
\begin{proof}
The assertion follows from standard estimates for the $L^2(\Gamma)$-projection 
using the regularity result from Lemma \ref{lem:regularity} as well as 
$|E|\sim h^2$.
\end{proof}
\begin{lemma}\label{lem:second}
 Assume that $y_d\in H^1(\Omega)$.
 Let $(y,u)$ be the optimal state and control solving \eqref{eq:optimality}.
 Then, there holds the error estimate
 \[
 \|y-B_h Q_h u\|_{L^2(\Omega)} \le c h^{\min\{2,-1+2\bar \lambda -2\varepsilon\}}
 \]
 with $\bar \lambda:=\min_{j\in \mathcal{C}}\lambda_j$ and any $\varepsilon>0$, provided that the refinement condition \eqref{eq:ref_cond} is satisfied.
\end{lemma}
\begin{proof}
In order to prove the assertion, we use a duality argument. Let $z\in H_0^1(\Omega)$ solve
\[
	-\Delta z=y-B_hQ_hu\text{ in }\Omega,\quad z=0\text{ on }\Gamma.
\]
Moreover, let $z_h\in V_{0h}$ denote its Ritz-projection. In the sequel, we first assume that $\Omega$ is convex. The non-convex case is discussed at the end of the proof. 
The integration by parts formula implies
\begin{equation}\label{eq:L2-1}
	\|y-B_hQ_hu\|_{L^2(\Omega)}^2=
	(Q_h u-u,\partial_n z)_{L^2(\Gamma)} + 
	(\nabla (y-B_hQ_hu),\nabla z)_{L^2(\Omega)}.
\end{equation}
Due to the convexity of $\Omega$, we have according to a standard trace theorem and elliptic regularity
\[
	\|\partial_n z\|_{H^{1/2}(\Gamma)}\leq c\|z\|_{H^{2}(\Omega)}\leq c\|y-B_hQ_hu\|_{L^2(\Omega)}.
\]
Consequently, by using the orthogonality of the $L^2$-projection and corresponding error estimates, we get for the first term on the right hand side of \eqref{eq:L2-1}
\[
	(Q_h u-u,\partial_n z)_{L^2(\Gamma)}=(Q_h u-u,\partial_n z-Q_h\partial_n z)_{L^2(\Gamma)}\leq c h^{2} \|u\|_{H^{1/2}(\Gamma)}\|y-B_hQ_hu\|_{L^2(\Omega)},
\]
where we note that $|E|\sim h^2$. Using the properties of the harmonic extensions, the Galerkin orthogonality of $z_h$, together with the fact that $I_hy-B_hI_hu$ and $(B_h-\tilde B_h)(I_h-Q_h)u$ belong both to $V_{0h}$ ($I_h$ denotes the standard Lagrange interpolant and $\tilde B_h$ the zero extension operator into $V_h$), we obtain for the second term in \eqref{eq:L2-1}
\begin{align*}
  (\nabla (y-B_hQ_hu),\nabla z)_{L^2(\Omega)}
  &=\left(\nabla (y-B_hQ_hu),\nabla (z-z_h)\right)_{L^2(\Omega)}\\
  &=\left(\nabla (y-I_hy +\tilde B_h(I_h-Q_h)u),\nabla (z-z_h)\right)_{L^2(\Omega)}\\
  &\le c\left(\|\nabla(y-I_hy)\|_{L^2(\Omega)}+\|\nabla\tilde B_h(I_h-Q_h)u\|_{L^2(\Omega)}\right)\|\nabla(z-z_h)\|_{L^2(\Omega)}.
\end{align*}
Note that the Lagrange interpolants of $y$ and $u$ are well-posed in the present case since both functions are continuous if the domain $\Omega$ is convex. By a standard estimate for the finite element error, we directly get
\[
	\|\nabla(z-z_h)\|_{L^2(\Omega)}\le c h \|y-B_hQ_hu\|_{L^2(\Omega)}.
\]
Since the grading towards the whole boundary implies a grading towards each corner, i.e.,
\[
	h_T\le
	\begin{cases}
		ch^2 & \text{if } \dist(T,\boldsymbol c_j)=0,\\
		ch(\dist(T,\boldsymbol c_j))^{1/2} & \text{if } \dist(T,\boldsymbol c_j)>0,\\
	\end{cases}
\]
we deduce by means of \eqref{eq:int_error_Sh} if $\dist(T,\boldsymbol c_j)=0$, and a standard interpolation error estimate if $\dist(T,\boldsymbol c_j)>0$ that
\[
	\|\nabla(y-I_hy)\|_{L^2(\Omega)}\le c h^{\min\{1,2-2\max_{j\in\mathcal{C}}\gamma_j\}} \|y\|_{V^{2,2}_{\vec \gamma}(\Omega)}\le c h^{\min\{1,-2+2\bar \lambda-2\varepsilon\}} \|y\|_{V^{2,2}_{\vec \gamma}(\Omega)},
\]
where we have set $\gamma_j=\max\{0,2-\lambda_j\}+\varepsilon$ with some arbitrary $\varepsilon>0$. Employing \eqref{eq:stability_tilde_Bh} and standard estimates for the $L^2$-projection and the Lagrange interpolant (after having introduced $u$ as an intermediate function), we get for $s = \min\{2,\bar\lambda\}-\varepsilon$ and any $\varepsilon>0$
\[
	\|\nabla\tilde B_h(I_h-Q_h)u\|_{L^2(\Omega)}\le ch^{-1}\|(I_h-Q_h)u\|_{L^2(\Gamma)}\le ch^{\min\{2,-2+2\bar\lambda\}-2\varepsilon}\|u\|_{H^{s-1/2}(\Gamma)},
\]
where we used $|E|\sim h^2$. 
Putting everything together, in combination with the regularity results of Lemma \ref{lem:regularity}, we have arrived at the assertion in case of convex domains.

In the non-convex case, by using the definition of very weak solutions,
\[
	(y,-\Delta v)_{L^2(\Omega)}=-(u,\partial_n v)_{L^2(\Gamma)}\quad \forall v\in \{v\in H^1_0(\Omega):\, \Delta v\in L^2(\Omega)\},
\]
and the
definition of $\partial^h_n$ \eqref{eq:def_var_normal}, we rewrite the error
term as follows:
\begin{align*}
	\|y-B_hQ_hu\|_{L^2(\Omega)}^2
	&=(y,y-B_hQ_hu)_{L^2(\Omega)}-(B_hQ_hu,y-B_hQ_hu)_{L^2(\Omega)}\\
	&=-(u,\partial_n z)_{L^2(\Gamma)}+(Q_hu,\partial_n^h z_h)_{L^2(\Gamma)}\\
	&=(u,\partial_n^hz_h-\partial_nz)_{L^2(\Gamma)},
\end{align*}
where we used that $B_h$ represents the discrete harmonic extension operator,
and the orthogonality of $Q_h$. Thus, the desired result follows in the
present case from the estimate stated in Theorem~\ref{thm:estimate_var_normal_deriv}, in which the term $\|y-B_hQ_h
  u\|_{L^2(\Omega)}$ appears on the right-hand side again.
\end{proof}
We now state the main result for the Dirichlet boundary control problem.
\begin{theorem}\label{thm:optimal_control}
  Let $y_d\in
  H^{1}(\Omega)$. If $\Omega$ is convex, assume 
  additionally $y_d\in C^{0,\sigma}(\overline\Omega)$, $\sigma\in (0,1)$. 
  Moreover, let $(y,u)$ and $(y_h,u_h)$ 
  denote the solutions of \eqref{eq:optimality} and \eqref{eq:discrete_opt_cond}, respectively.
  If the sequence of computational meshes satisfies the refinement condition \eqref{eq:ref_cond}, the estimate
  \begin{align*}
     \|u-u_h\|_{L^2(\Gamma)} + \|y-y_h\|_{L^2(\Omega)} \le c h^{\min\{2,-1+2\bar\lambda-2\varepsilon\}}\lnh^{3/2}
  \end{align*}
  is valid with $\bar\lambda :=\min_{j\in\C}\lambda_j$ and any $\varepsilon>0$.
\end{theorem}
\begin{proof}
	Due to \eqref{eq:basic_estimate} and \eqref{eq:adjointnormal}, the result is a consequence of Lemmas \ref{lem:first_term} and \ref{lem:second}, and Theorem~\ref{thm:estimate_var_normal_deriv}.
\end{proof}

\subsection{Numerical experiments}\label{sec:experiments1}
The following experiments are similar to those from Section~\ref{sec:experiments}.
We consider the domains $\Omega_\omega$ with $\omega\in\{2\pi/3,3\pi/4,3\pi/2\}$.
The largest singular exponent is denoted by $\bar\lambda:=\pi/\omega$.
Using polar coordinates $(r,\varphi)$ located at the origin, the exact solution of our benchmark problem is set to
\begin{align*}
	y(x_1,x_2) &:= -\bar\lambda r^{\bar\lambda-1}(x_1,x_2)(1-x_1^2)(1-x_2^2) + 2 r^{\bar\lambda}(x_1,x_2) \sin(\lambda \varphi(x_1,x_2))(x_1^2+x_2^2-2)\\
	p(x_1,x_2) &:= r^{\bar\lambda}(x_1,x_2) \sin(\bar\lambda\varphi(x_1,x_2))(1-x_1^2)(1-x_2^2).
\end{align*}
Note, that the function $p$ fulfills homogeneous Dirichlet boundary conditions.
The function $y$ is not harmonic and hence, we consider instead the state equation
\[
	-\Delta y = f\quad\mbox{in}\quad \Omega
\]
with some $f$ which can be computed by means of $y$. The desired state $y_d$ can be computed from the adjoint equation taking into account the definitions of $p$ and $y$.
With a simple computation we easily confirm that the optimality condition $u=\alpha^{-1} \partial_n p$ is fulfilled. In this experiment the regularization parameter is chosen to satisfy $\alpha=1$.
Note that we considered $f\equiv 0$ in the theory, but the results derived in Theorem \ref{thm:optimal_control} hold for the inhomogeneous case as well. The meshes are reused from the experiments in Section~\ref{sec:experiments}, see also Figure~\ref{fig:meshes}.
The optimality condition of the discretized problem, more precisely the equation
\[
	(\alpha u_h - \partial_n^h p_h,w_h)_{L^2(\Gamma)} = 0\qquad \forall w_h\in V_h^\partial,
\]
with $p_h\in V_{0h}$ as the solution of
\[
	\begin{aligned}
		p_h\in V_{0h}\colon &&(\nabla v_h,\nabla p_h)_{L^2(\Omega)} &= (y_h-y_d,v_h)_{L^2(\Omega)} && \forall v_h\in V_{0h},\\
		y_h\in V_h\colon &y_h= u_h\ \mbox{on}\ \Gamma, & (\nabla y_h,\nabla v_h)_{L^2(\Omega)} &= (f,v_h)_{L^2(\Omega)} && \forall v_h\in V_{0h},
	\end{aligned}
\]
has been solved with the GMRES method. Moreover, the linear solver MUMPS has been used to compute $y_h$ from $u_h$ and $p_h$ from $y_h$.

The results of the numerical tests are summarized in Table \ref{tab:experiment} for $\omega\in\{2\pi/3, 3\pi/4, 3\pi/2\}$.
These experiments confirm that the discrete controls converge with the rate $2$ when the interior angles are less than $120^\circ$.
For larger angles the convergence rate is reduced. For $\omega=3\pi/4$ and $\omega=3\pi/2$, we have proven a rate close to $5/3$ and $1/3$, respectively. The observed convergence rates are in agreement with the predicted ones. As often in optimal control, in case that full order of convergence is no longer achievable by the discrete controls, the discrete states still converge with a higher rate than predicted by the theory derived via the optimality conditions. For similar observations, we also refer to \cite{APR13,APW14} in case of Neumann control problems and \cite{MRV13,AMPR16} in case of Dirichlet boundary control problems.
A comparison of the error propagation between quasi-uniform meshes and meshes with boundary refinement is illustrated in Figure~\ref{fig:eoc}, also for the domains with $\omega\in\{\pi/2, 5\pi/4, 7\pi/4\}$.
For a sufficiently fine initial mesh, the error is always smaller for boundary concentrated meshes.

\begin{table}[htbp]
	\setlength{\tabcolsep}{4pt}
	\begin{tabular}{rrrrrrr}
		\toprule
		& \multicolumn{2}{c}{$\omega=120^\circ$} & \multicolumn{2}{c}{$\omega=135^\circ$} & \multicolumn{2}{c}{$\omega=270^\circ$} \\
		\cmidrule(lr{1em}){2-3} 
		\cmidrule(lr{1em}){4-5} 
		\cmidrule(lr{1em}){6-7} 
		\multicolumn{1}{l}{$N$} & \multicolumn{1}{l}{$\|u-u_h\|_{L^2(\Gamma)}$} & \multicolumn{1}{l}{$\|y-y_h\|_{L^2(\Omega)}$} & \multicolumn{1}{l}{$\|u-u_h\|_{L^2(\Gamma)}$} & \multicolumn{1}{l}{$\|y-y_h\|_{L^2(\Omega)}$} & \multicolumn{1}{l}{$\|u-u_h\|_{L^2(\Gamma)}$} & \multicolumn{1}{l}{$\|y-y_h\|_{L^2(\Omega)}$} \\ 
		\midrule
		$6$ & 4.97e-2 (1.90) & 6.07e-3 (2.64) & 6.26e-2 (1.87) & 7.38e-3 (2.67) & 4.69e-1 (0.28) & 2.08e-1 (0.67) \\ 
		$8$ & 1.27e-2 (1.97) & 1.04e-3 (2.54) & 1.64e-2 (1.93) & 1.19e-3 (2.63) & 3.93e-1 (0.25) & 1.33e-1 (0.65) \\ 
		$10$ & 3.24e-3 (1.97) & 2.26e-4 (2.20) & 4.31e-3 (1.93) & 2.40e-4 (2.31) & 3.23e-1 (0.28) & 8.45e-2 (0.65) \\ 
		$12$ & 8.13e-4 (2.00) & 5.15e-5 (2.14) & 1.14e-3 (1.92) & 5.30e-5 (2.18) & 2.62e-1 (0.30) & 5.37e-2 (0.65) \\ 
		$14$ & 2.05e-4 (1.99) & 1.29e-5 (2.00) & 3.08e-4 (1.89) & 1.31e-5 (2.02) & 2.10e-1 (0.32) & 3.40e-2 (0.66) \\ 
		$16$ & 5.15e-5 (1.99) & 5.15e-5 (1.99) & 8.52e-5 (1.85) & 3.24e-6 (2.01) & 1.68e-1 (0.32) & 2.15e-2 (0.66) \\ 
		\bottomrule
	\end{tabular}
	\caption{Absolute error and experimental convergence rate for the discrete solution of \eqref{eq:discrete_opt_cond}, for the domains $\Omega_\omega$, $\omega\in\{2\pi/3, 3\pi/4, 3\pi/2\}$.}
	\label{tab:experiment}
\end{table}

\begin{figure}[htb]
	\begin{center}
		\includegraphics[width=.7\textwidth]{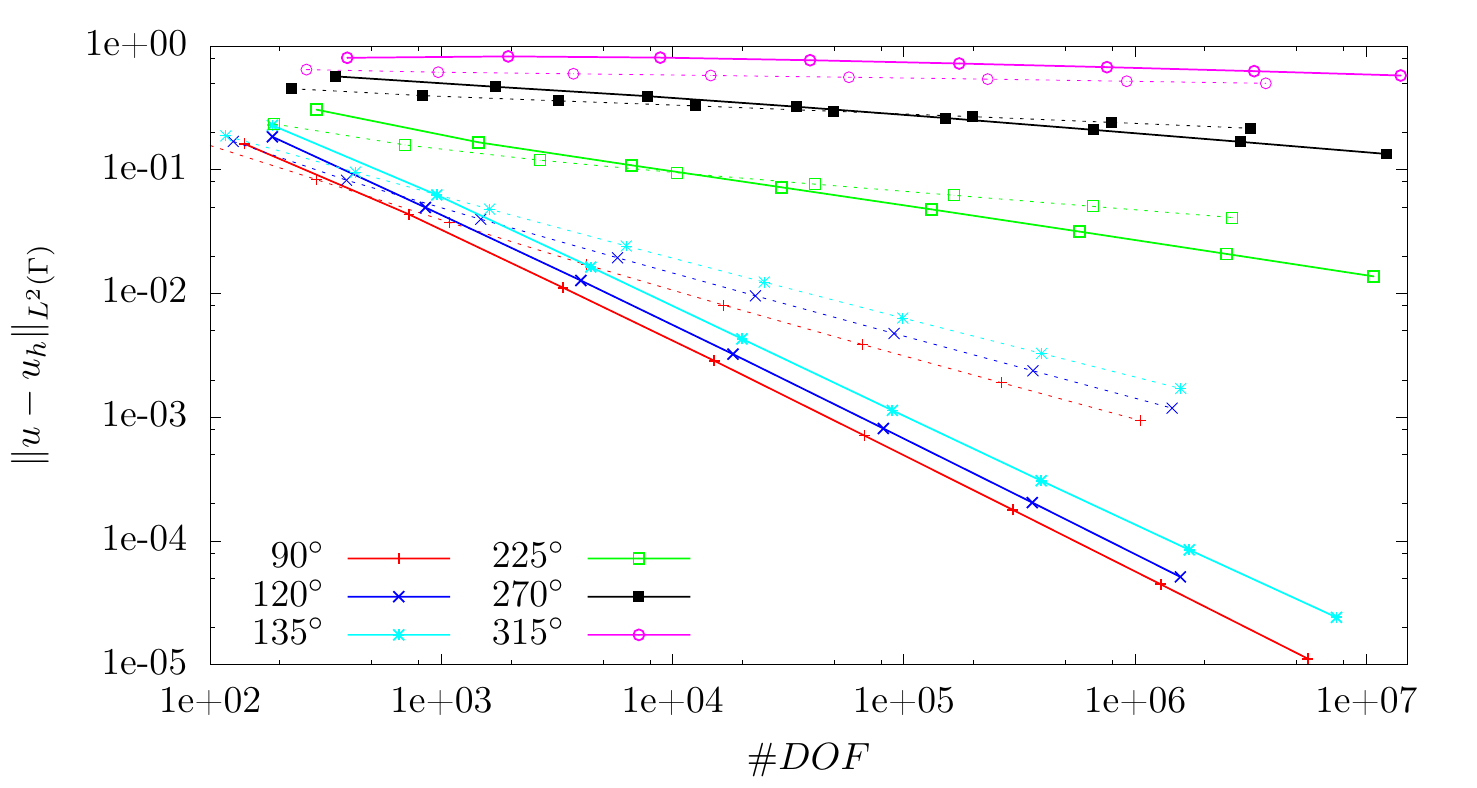}
	\end{center} 
	\caption{Convergence plot illustrating the approximation error on several domains for quasi-uniform (dashed lines) and boundary concentrated meshes (solid lines).}
	\label{fig:eoc}
\end{figure}

\bibliographystyle{plain}
\bibliography{bibliography2}
\end{document}